\documentclass[12pt,a4paper, reqno]{amsart}
\usepackage{float} 
\usepackage[utf8]{inputenc} 
\usepackage{tcolorbox} 
\usepackage{dsfont} 
\usepackage{hyperref} 
\usepackage{amsmath} 
\usepackage{amssymb} 
\usepackage{mathtools} 
\usepackage{algorithmic} 
\usepackage[ruled,vlined]{algorithm2e} 
\usepackage{amsfonts,amsthm}
\usepackage[foot]{amsaddr} 
\usepackage[margin=2.9cm]{geometry} 
\usepackage{enumitem}  

\theoremstyle{plain}
\newtheorem{Th}{Theorem}[section]
\newtheorem{lemma}[Th]{Lemma}
\newtheorem{Cor}[Th]{Corollary}
\newtheorem{proposition}[Th]{Proposition}

\theoremstyle{definition}
\newtheorem{definition}[Th]{Definition}

\newtheorem{Rem}[Th]{Remark}
\newtheorem{?}[Th]{Problem}

\DeclareMathOperator*{\argmin}{arg\,min}
\newcommand{\supp}{\textup{supp }}

\usepackage{soul}
\newcommand{\alex}[1]{\textcolor{black}{#1}}
\newcommand{\leave}[1]{}

\usepackage{fancyhdr}
\providecommand{\keywords}[1]
{
  \small	
  \textbf{\textit{Keywords}} 
}

\begin{document}

\title{Multi-component separation, inpainting and denoising with recovery guarantees}
\author{Van Tiep Do$^{* \dagger}$ }

\keywords{ Geometric Separation \and wavelets \and Shearlets \and   $l_1$-minimization \and  Sparsity \and Cluster Coherence \and dual frames.}


\address{ $^{*}$ Department of Mathematics, Technische Universit{\"a}t Berlin, 10623 Berlin, Germany}

\address{ $^{\dagger}$ Vietnam National University, 334 Nguyen Trai, Thanh Xuan, Hanoi}

\email{tiepdv@math.tu-berlin.de}
\thanks{}


\dedicatory{}

\maketitle

\begin{abstract}
In image processing, problems of separation and reconstruction of missing pixels
from incomplete digital images have been far more advanced in past decades.  Many empirical results   have produced very good results, however, providing a theoretical analysis for the success of algorithms is not an easy task, especially, for inpainting and separating multi-component signals. In this paper,  we propose two main algorithms based on $l_1$ constrained  and unconstrained minimization for 
 separating $N$ distinct geometric components and simultaneously filling-in the missing part of the observed image. We then present a theoretical guarantee for these algorithms using compressed sensing technique, which is based on a principle that each component can be sparsely represented by a suitably chosen dictionary. Those sparsifying systems are extended to the case of general frames instead of Parseval frames which have been typically used in the past.  We finally prove that the method does
indeed succeed in separating point singularities from curvilinear singularities and texture as well as inpainting the missing band contained in curvilinear singularities and texture. 
\end{abstract}

\section{Introduction} 
In practical applications, many problems in signal or image processing involve the task of  separating two or more components from a given image. This problem is of interest in many applications \cite{30,31}, for instance, astronomers may desire to decompose the observed image into pointlike, curvelike, and texture structures. This raises a need to separate  various components from the original image in order to effectively analyze them. 
Also, the task of inpainting an image which is often composed of two or more components is  of interest in  many research fields.
However, at the first sight the above problems seem unable to solve from a theoretical point of view   because of the ill-poseness of the system that leads to infinitely many solutions. Recently, compressed sensing technique provides a firm foundation to recover a signal from incomplete measurements by exploiting the sparsity of the signal \cite{5}. It is well known that $l_1$ minimization performs exceptionally well in locating sparse solutions of underdetermined linear systems.
 Based on this observation, a great number of sparsity-based algorithms have been \alex{introduced in} this area. Various empirical works in the past have shown good results
 for either inpainting  \cite{3} or geometric separation \cite{20}. \alex{The} theoretical study \cite{12} first introduced an asymptotic result for separating points from curves using Parseval frames of wavelets and curvelets.
It was then followed by some other works to either inpaint a single component image \cite{7,15,18} or separate two components \cite{11,12,14}. Later, our previous work \cite{14} proposed a framework to  deal with simultaneous separation and inpainting for two-component signals. 
  A critical premise of these methods is that each of the components has a sparse representation in a certain dictionary.  
  
However, all aforementioned papers used Parseval frames and none of them deal with the case of multi-component separation and inpainting. The \alex{problem studied here} is completely different since we have more intersections among components which makes the problem more challenging. Each underlying layer has distinctive features, for instance, pointlike structure (smooth apart from point singularities), cartoon (the piecewise smooth part of the image), and texture (consisting of repeated texels). 
 There is still a lack of theoretical analysis  to explain why simultaneous separation and inpainting is possible for multi-component signals. In an attempt to generalize the theory to multi-component case, the proceeding paper \cite{10} introduced a theory which also exploited notions of  joint concentration and cluster coherence that first appeared in \cite{12}. However, the theory in \cite{10} can be problematic to apply for problems of separating more than two components,  which will be discussed in Section \ref{0208-7}. 

In this paper, motivated by our previous works \cite{11,14} we present a more general framework for simultaneous separation, inpainting, and denoising on degraded multi-component signals.  We introduce two algorithms based on $l_1$ minimization for analyzing these tasks, one deals with noiseless images and the other with additional noise. We then provide a theoretical guarantee for these algorithms by establishing $L_2$ error of the reconstructions.  In addition, our methodology  works for non-Parseval frames if  each of them possesses a well localized dual frame, see \cite{11} for the case of separating two components.  \alex{We further} apply the results to \alex{the separation of} points from curves and texture while inpaint\alex{ing} the missing content of the image. The key idea to the success of our method lies in the morphological difference among these ground truth components. They are then encoded in three different sparse representation systems, wavelets for representing point singularities, shearlets for representing curvilinear singularities, and Gabor frames for representing texture.  Based on that, we prove the success of the proposed algorithms using a multi-scale approach.

\subsection{Outline}
The remainder of this paper is organized  as follows.  We first start with some basic definitions, notions, and the formulation  of inpainting and geometric separation problems in Section \ref{0208-5}. Next, we modify the notions of joint concentration and cluster coherence to adapt to our need in Section \ref{0208-7}. We then derive a general theoretical guarantee in Theorem \ref{2006-3} using general frames for separating $N-$ component signals based on $l_1$ minimization. In section \ref{0208-6}, a similar theoretical guarantee is concluded in Theorem \ref{2006-4} to deal with noisy images. We finally transfer the abstract theory to derive Theorem \ref{3007-10} and \ref{3007-11} which prove that our algorithms are capable of denoising, inpainting, and separating points from curves and texture  in Section \ref{0208-10}.  The paper ends with a summary and conclusion in Section \ref{0208-1}.

\section{Mathematical formulation} \label{0208-5}
We  start with some notions and basic definitions.
 \subsection{Notations and defintions}
Let us first recall the  \emph{Schwartz space} of    \emph{rapidly decreasing functions} or \emph{Schwartz functions} 
\begin{equation} \label{0607-1}
     \mathcal{S}(\mathbb{R}^2):= \Big \{ f \in C^\infty(\mathbb{R}^2)  \mid   \forall K, N \in \mathbb{N}_0, \sup_{x \in \mathbb{R}^2} (1 + |x|^2 )^{-\frac{N}{2}} \sum_{|\alpha |\leq K} | D^\alpha f(x) | < \infty \Big \}. 
\end{equation}
The Fourier transform and inverse Fourier transform  for $f,F \in  \mathcal{S}(\mathbb{R}^2)$ are then defined  by 
$$ \hat{f}(\xi)=\mathcal{F}[f](\xi)= \int_{\mathbb{R}^2} f(x) e^{-2\pi i  x^T \xi} dx,$$
$$ \check{F}(x) =\mathcal{F}^{-1}[F](x)= \int_{\mathbb{R}^2} F(\xi) e^{2\pi i \xi^T x} dx,$$
which can be extended to functions in $L^2(\mathbb{R}^2)$, see \cite{13}.

Next, we recall some needed
definitions and facts from frame theory, more details can be found in \cite{4,6}. We call a family of vector $\Phi = \{ \phi_i \}_{i \in I}$ in a separable Hilbert space $\mathcal{H}$  a \emph{frame} for $\mathcal{H}$ if there are constants $A, B >0 $ so that for all $f \in \mathcal{H}$ we have
\begin{equation} \label{2207}
    A \| f \|_2^2 \leq \sum_{i \in I} | \langle f, \phi_i \rangle |^2 \leq B \|f \|_2^2, 
\end{equation} 
where $A, B$ are called the \emph{lower} and \emph{upper frame bound}.  We call it \emph{$A-$tight frame} in case  $A = B$ and a \emph{Parseval frame} or  \emph{tight frame} in case $A=B=1$.

By abuse of notation we write $\Phi$ again to denote the \emph{synthesis operator}
$$ \Phi : l_2(I) \rightarrow \mathcal{H}, \quad \Phi(\{ c_i \}_{i \in I} ) = \sum_{i \in I }c_i \phi_i.$$
We  define the \emph{analysis operator} $\Phi^*$ by
$$ \Phi^* : \mathcal{H}  \rightarrow l_2(I), \quad \Phi^*(f ) = (\langle f, \phi_i \rangle)_{i \in I} .$$
Associated with a frame $\Phi = \{ \phi_i \}_{i \in I},$ we define  the \emph{frame operator} $\mathbb{S}$ by
$$\mathbb{S}=\Phi \Phi^*: \mathcal{H} \rightarrow \mathcal{H}, \quad \mathbb{S} f= \sum_{i \in I} \langle f, \phi_i \rangle \phi_i. $$
Given a frame $\Phi = \{ \phi_i 
\}_{i \in I}$  for $\mathcal{H}$,  a sequence $\Phi^d = \{ \phi^d_i \}_{i \in I}$ in $\mathcal{H}$ is called \\  \phantom{11} - an \emph{analysis pseudo-dual} if the following holds
\begin{equation} \label{2207-1}
   f=\Phi (\Phi^d)^* f= \sum_{i \in I}\langle f, \phi^d_i \rangle \phi_i, \; \forall f \in \mathcal{H};
\end{equation}
    \quad - a \emph{synthesis pseudo-dual} if we have
    \begin{equation} \label{2207-2}
    f=(\Phi^d) \Phi^* f= \sum_{i \in I}\langle f, \phi_i \rangle \phi_i^d,  \; \forall f \in \mathcal{H};
    \end{equation}
 \quad - a \emph{dual frame} of $\Phi$ if both  \eqref{2207-1} and \eqref{2207-2} hold.
By \eqref{2207},  the corresponding frame operator $\mathbb{S}$ is
self-adjoint, invertible on $\mathcal{H}$ \cite{6}.
As a result, there exists a dual frame 
$\{ \mathbb{S}^{-1} \phi_i \}_{i \in I}$ 
called the \emph{canonical
dual frame} of $\Phi = \{ \phi_i \}_{i \in I}$  with frame bounds $(B^{-1}, A^{-1})$.

Let $\mathbb{D}_{\Phi}$ denote the set of all synthesis pseudo-dual of the frame $\Phi$
\begin{equation} \label{Eq2}
    \mathbb{D}_{\Phi}= \{ \Phi^d = \{ \phi^d_i \}_{i \in I} \in \mathcal{H} \mid f=\Phi^d \Phi^* f= \sum_{i \in I}\langle f, \phi_i \rangle \phi_i^d, \;\forall f \in \mathcal{H} \}.
\end{equation} 
Clearly, $\mathbb{D}_{\Phi} \neq \emptyset$ since 
   $\mathbb{S}^{-1}\Phi \in \mathbb{D}_{\Phi}$. In addition, in case  $\Phi $ forms a Parseval frame  
we have $\Phi \in \mathbb{D}_\Phi$.

 \subsection{Problem of image separation}

In image processing, the image to be analyzed is often composed of multiple components. Given an image $f$, we assume that $f$ can be decomposed of $N$ components, i.e., 
\begin{equation} \label{0608-1}
    f= \sum_{m}f_m,
\end{equation} where \alex{each} $f_m$ corresponds to \alex{a} geometric feature contained in  $f$. It would be ideal to extract all of these \alex{parts} from the original image for analyzing them separately.  
For practical purposes, we later assume that these components contain  texture and jumps or changes in behavior along point, curvilinear discontinuities.
The prior information on underlying components is important to capture the geometry of the discontinuities by choosing suitable dictionaries, each sparsely represent\alex{ing} its intended content type. 

\subsection{Problem of image  inpainting}
Image inpainting \leave{involves}\alex{is} the task of filling in holes or removing  selected objects contained in an image using information from the surrounding area. Given a Hilbert space $\mathcal{H}$, we assume that $\mathcal{H}$ can be decomposed into a direct sum of missing subspace and known subspace, i.e.,  $\mathcal{H} = \mathcal{H}_K \oplus \mathcal{H}_M$. Let $P_M$ and $P_K$ denote the orthogonal projection of $\mathcal{H}$ onto the missing part and the known part, respectively.  Our goal is to reconstruct the original image $f$ from knowledge of the known part $P_Kf.$
\subsection{Recovery via $l_1-$ minimization }
Among existing methods, $l_1$ minimization has been widely used for recovering sparse solutions. One appeared in several papers \cite{2,23} for separating two geometric components of a given image. There, the following procedure was proposed  for an observed signal $f$ and two Parseval frames $\{ \Phi_1 \}_{i \in I}, \{ \Phi_2 \}_{j \in J} $, 
    \begin{eqnarray}
    (f_1^{\star}, f_2^{\star}) &= & \argmin_{f_1, f_2} \| \Phi_1^* f_1 \|_1 + \| \Phi_2^* f_2 \|_1, 
     \quad \text{subject to} \quad f_1 + f_2 = f. \label{AL0-3007}
   \end{eqnarray}
The theoretical guarantee for \eqref{AL0-3007} presented in the aforementioned papers was based on notions of joint concentration and cluster coherence. 
 In our analysis, we consider the following  algorithm for simultaneously inpainting and separating $N$ distinct constituents from an observed image. 
\vskip 1mm
\begin{algorithm}[H]
\SetAlgoLined
\KwData{observed image $f$, $N$  frames $\{ \Phi_{1 i_1} \}_{i_1 \in I_1}, \dots, \{ \Phi_{N i_N} \}_{i_N\in I_N}$ with frame bounds $(A_1, B_1),$ $(A_2, B_2), \dots, (A_N, B_N)$.}

\textbf{Compute:} $(f_1^\star, f_2^\star, \dots, f_N^\star)$, where 
\begin{eqnarray*}
    (f_1^{\star}, f_2^{\star},\dots, f_N^\star ) &= & \argmin_{(f_1, \dots, f_N)} \| \Phi_1^* f_1 \|_1 + \| \Phi_2^* f_2 \|_1 + \dots \| \Phi_N^* f_N \|_1, \nonumber \\
    && \quad \text{subsect to} \quad P_K(f_1 + f_2 + \dots + f_N)= P_Kf.
   \end{eqnarray*}  \vskip 0.5mm
\KwResult{ recovered components $f_1^{\star}, f_2^{\star}, \dots, f_N^\star.$ } 
\vskip 1mm
 \caption{\label{AL1-3007} Constrained optimization}
\end{algorithm}
\vskip 1mm
We then provide a theoretical guarantee for the success of the algorithm in the next section.

\section{Theoretical guarantee for multiple component separation problems} \label{0208-7}
\subsection{Joint concentration analysis}

The notion of joint concentration was first introduced in \cite{12} for separating  two geometric components using Paseval frames.  There, the  joint concentration associated with two Parseval frames $\Phi_1, \Phi_2$ and  sets of indexes $\Lambda_1, \Lambda_2$ was defined by
\begin{equation} \label{0508-1}
    \kappa(\Lambda_1, \Lambda_2)= \sup_{x \in \mathcal{H}} \frac{\| \mathds{1}_{\Lambda_1}\Phi_1^\star x \|_1 + \| \mathds{1}_{\Lambda_2} \Phi_2^\star x\|_1  }{\| \Phi_1^\star x \|_1 + \|\Phi_2^\star x\|_1}.
\end{equation} 
In our situation, we are now given the known part $P_K f$ of a signal $f$ where  $f$ is a superposition of $N$ geometric constituents. 
  An approach was proposed in \cite{10} to deal with  separating tasks, i.e., $P_K f =f,$ there
\begin{equation} \label{1108-1}
    \kappa = \sup_{x \in \mathcal{H}} \frac{\| \mathds{1}_{\Lambda_1}\Phi_1^* x \|_1 + \| \mathds{1}_{\Lambda_2} \Phi_2^* x\|_1  + \dots + \|\mathds{1}_{\Lambda_N}\Phi_N^* x \|}{\| \Phi_1^* x \|_1 + \|\Phi_2^* x\|_1 + \dots + \|\Phi_N^*x\|_1}.
\end{equation} 
Let us now have a closer look to these definitions.
Let $(f_1^*, f_2^*, \dots, f_N^*) $ be the solution of Algorithm \ref{AL1-3007} with $P_K = Id,$ and $(f_1, f_2, \dots, f_N)$ be the truly underlying components.
In case $N=2$, the underlying reason for the definition \eqref{0508-1} lies in the fact that $ f_1^* - f_1=-(f_2^* - f_2 )$ and the absolute value of the error is then encoded into variable  $x$ in   \eqref{0508-1}. It seems plausible that \eqref{1108-1} is a generalized version of \eqref{0508-1}.
However, the approach can be problematic  for more components   since from assumption $f_1^\star + f_2^\star + f \dots + f_N^\star =f_1 + f_2 +  \dots + f_N, N \geq 3$  we can not invoke a similar relation to encode with the same variable $x$ as in \eqref{1108-1}.  

In our study, we  modify the definition of joint concentration and cluster coherence to adapt to the situation of simultaneous separation and inpainting.
 Obviously, the separation problem is a trivial case of simultaneous separation and inpainting problem as we can set missing space $\mathcal{H}_M = \emptyset$.
Since we just know the known part, i.e., $P_K f = P_K(f_1^\star + f_2^\star + f \dots + f_N^\star) = P_K(f_1 + f_2 +  \dots + f_N )$ we will encode with different terms corresponding with  component errors $f_i^* - f_i, 1 \leq i \leq N .$
This will be now made precise.

\begin{definition} Let $\Phi_1, \Phi_2, \dots, \Phi_N$ be a set of  frames with frame bounds $(A_1, B_1),$ $(A_2, B_2), \dots, (A_N, B_N),$ and $\Lambda_1, \Lambda_2, \cdots, \Lambda_N $ be index sets. We define
\begin{enumerate}[label=(\roman*)]
    \item  the  \emph{N-joint concentration} $\kappa:= \kappa(\Lambda_1, \Phi_1; \dots;  \Lambda_N, \Phi_N)$ by 
\begin{equation} \label{CT0} 
 \kappa_N = \sup_{\sum_{n=1}^N f_N \in \mathcal{H}_M} \frac{\sum_{m} \Big (\| \mathds{1}_{\Lambda_m}\Phi_m^* P_K f_m \|_1 + \| \mathds{1}_{\Lambda_m}\Phi_m^* P_M f_m \|_1 \Big )}{\| \Phi_1^* f_1 \|_1 + \dots + \|\Phi_N^*f_N\|_1}, 
\end{equation} 
\item  the  \emph{separation N-joint concentration} $\kappa_N^{sep}:= \kappa(\Lambda_1, P_K\Phi_1; \dots;  P_K\Lambda_N, \Phi_N)$ by
\begin{equation} \label{CT1} 
\kappa_N^{sep}=  \sup_{\sum_{n=1}^N f_n \in \mathcal{H}_M} \frac{\| \mathds{1}_{\Lambda_1} \Phi_1^* P_Kf_1 \|_1 + \dots + \|\mathds{1}_{\Lambda_N} \Phi_N^* P_Kf_N \|_1}{\| \Phi_1^* f_1 \|_1 + \dots + \|\Phi_N^*f_N\|_1},
\end{equation}
\item 
  the  \emph{inpainting N-joint concentration} $\kappa_N^{inp}:= \kappa(\Lambda_1, P_M\Phi_1; \dots;  \Lambda_N, P_M\Phi_N)$ by
\begin{equation} \label{CT2}
\kappa_N^{inp}=  \sup_{(f_1, f_2, \dots, f_N)} \frac{\| \mathds{1}_{\Lambda_1}\Phi_1^* P_M f_1 \|_1 +   \dots + \|\mathds{1}_{\Lambda_N}\Phi_N^* P_M f_N \|_1}{\| \Phi_1^* f_1 \|_1 + \dots + \|\Phi_N^*f_N\|_1}.
\end{equation} 
\end{enumerate}
\end{definition}

These notions enable us to provide a theoretical for the success of  Algorithm \ref{AL1-3007}. Let us dive deeper into these notations by considering some special cases. If we consider each problem separately, for instance, in case of only separation problem, i.e., $\mathcal{H}_M=\emptyset,$ the joint concentration $\kappa_N \equiv \kappa_N^{sep}$, 
$\kappa_N^{sep}$ is therefore used to be responsible for the separation task. In this case, if $N=2$, the term $\kappa_2^{sep}$ is coincided with \eqref{0508-1} which guarantees the success of separating two components. 
 In case of without separating, only inpainting, i.e., $\kappa \equiv \kappa_1^{inp} =  \sup_{f_1} \frac{\| \mathds{1}_{\Lambda_1}\Phi_1^* P_M f_1 \|_1 }{\| \Phi_1^* f_1 \|_1 },$ this notation is slightly different from one that appeared in \cite{15} where $f_1$ was assumed to be in missing space $\mathcal{H}_M$, but they are bounded by the same cluster coherence which is later shown in Lemma \ref{Lm1} and they therefore play a similar role in the theoretical analysis of inpainting. 
Obviously, $\kappa_N \leq \kappa_N^{sep} + \kappa_N^{inp}.$ Bounding each term ensures the successful algorithm for the corresponding task. 

Next, we also extend a definition from \cite{15} to the case of general frames instead of Parseval frames.
\begin{definition} 
Fix $\delta > 0$. Given a Hilbert space $\mathcal{H}$ with a  frame $\Phi, f \in \mathcal{H} $ is $\delta-$relatively sparse in $\Phi$ with respect to $\Lambda$ if $\| \mathds{1}_{\Lambda^c} \Phi^\star f \|_1 \leq \delta , $ where $A^c $ denotes $X\setminus A.$
\end{definition}

The following result gives an estimate for the $L_2$ error of the reconstructed signals by Algorithm \ref{AL1-3007}.
\begin{proposition} \label{Pr1}
Let $\Phi_1, \Phi_2, \dots, \Phi_N$ be a set of  frames with pairs of frame bounds $(A_1, B_1),$ $(A_2, B_2), \dots, (A_N, B_N).$ For $  \delta_1, \delta_2, \dots, \delta_N \allowbreak > 0,$ we fix $\delta= \sum_{m=1}^N\delta_m,$ and suppose that $f \in \mathcal{H}$ can be decomposed as $x= \sum_{m=1}^N f_m$ so that each component $f_1, f_2, \dots, f_N$ is $\delta_1, \delta_2, \dots, \delta_N-$relatively sparse in $\Phi_1, \Phi_2, \dots, \Phi_N$ with respect to $\Lambda_1, \Lambda_2, \cdots, \Lambda_N,$ respectively. Let $(f_1^\star, f_2^\star, \dots, f_N^\star)$ solve Algorithm \ref{AL1-3007}. If we have $ \kappa_N < \frac{1}{2}$, then  
\begin{equation} \label{CT20}
\sum_{m=1}^N \| f_m^\star - f_m \|_2 \leq \frac{2\delta}{1 - 2 \kappa_N}.
\end{equation}
 
\end{proposition} 
\begin{proof}
First, we set $z_m= f^\star_m - f_m, $ for $1 \leq m \leq N$. Since $\Phi_m$ is a Parseval frame for each $m \in \{ 1,2, \dots, N \}$, thus
\begin{eqnarray} 
\sum_{m=1}^N \|  f_m^\star - f_m \|_2 &=&  \sum_{m=1}^N  \| \Phi_m^*(f^\star_m - f_m ) \|_2 \nonumber \\
&\leq &  \sum_{i=1}^N \| \Phi_m^*(z_m) \|_1  : = \mathbb{T}. \label{A1} 
\end{eqnarray}
We observe that $P_K(\sum_{m=1}^Nf^\star_m)   = P_K(\sum_{i=1}^N f_m) =P_Kf$, this leads to 
\begin{equation}
  \sum_{i=1}^N z_m \in \mathcal{H}_M.   \label{A2}
\end{equation}
By \eqref{A2} and triangle inequality, we have
\begin{equation}
    \sum_{m=1}^N\| \mathds{1}_{\Lambda_m} \Phi_m^* z_m \|_1 \leq \sum_{m=1}^N\| \mathds{1}_{\Lambda_m} \Phi_m^* P_Kz_m \|_1  + \sum_{m=1}^N\| \mathds{1}_{\Lambda_m} \Phi_m^* P_Mz_m \|_1 \leq \kappa_N \mathbb{T}. \label{A3}
\end{equation}
Thus, we obtain
\begin{eqnarray}
\mathbb{T} &=& \sum_{m=1}^N\| \mathds{1}_{\Lambda_m} \Phi_m^* z_m \|_1 +\sum_{m=1}^N \| \mathds{1}_{\Lambda_m^c} \Phi_m^* (f^\star_m - f_m) \|_1 \nonumber \\
&\stackrel{\mathclap{\normalfont{ \textup{By} \; (\ref{A3})}} } \leq &
\quad \kappa_N \mathbb{T} +\sum_{m=1}^N \| \mathds{1}_{\Lambda_m^c} \Phi_m^* f^\star_m \|_1 + \delta \nonumber \\
&=& \kappa_N \mathbb{T} +\sum_{i=1}^N (\|  \Phi_m^*  f_m^\star \|_1 -  \| \mathds{1}_{\Lambda_m} \Phi_m^*  f_m^\star \|_1) + \delta.
\label{A4}
\end{eqnarray}
We note that $(f_1^\star, f^\star_2, \dots, f_N^\star) $ is a minimizer of $\textup{Algorithm \ref{AL1-3007}}$. Thus
\begin{equation}
\sum_{m=1}^N \| \Phi_m^* f^\star_m \|_1 \leq  \sum_{m=1}^N \| \Phi_m^* f_m \|_1 \label{A5}.
\end{equation}
Using \eqref{A3}, \eqref{A4}, \eqref{A5} and triangle inequality, we have
\begin{eqnarray*}
\mathbb{T} & \leq  &  \kappa_N \mathbb{T} +\sum_{i=1}^N  \|\Phi_m^* f_m \| + \sum_{i=1}^N (\| \mathds{1}_{\Lambda_m} \Phi_m^*  z_m \|_1 -  \| \mathds{1}_{\Lambda_m} \Phi_m^*  f_m \|_1) + \delta \\
& \leq &  \kappa_N \mathbb{T} + \sum_{i=1}^N \| \mathds{1}_{\Lambda_m} \Phi_m^*  z_m \|_1 + 2 \delta \\
& \leq &  2\kappa_N\mathbb{T} + 2 \delta.
\end{eqnarray*} 
Thus, we finally obtain $$\mathbb{T} \leq \frac{2 \delta}{1 - 2  \kappa_N }.$$
\end{proof}

For real world images, geometric constituents of the image $f$  can be supported at different places. We often deal with the case in which a component does not have missing values. This  appears more clearly from the following remark.
\begin{Rem} \label{3007-20}
For fixed $i \in \{ 1,2, \dots,N\}$ if there is no missing at component $f_i$, the definition of the joint \alex{concentration} is modified \alex{corresponding} to $P_K f_i = f_i, P_Mf_i =0,$ i.e.,
\begin{equation*} 
 \kappa_N = \sup_{\sum\limits_{n=1}^N f_N \in \mathcal{H}_M} \frac{\| \mathds{1}_{\Lambda_i}\Phi_i^* f_i \|_1 + \sum\limits_{m \neq i} \Big (\| \mathds{1}_{\Lambda_m}\Phi_m^* P_K f_m \|_1 + \| \mathds{1}_{\Lambda_m}\Phi_m^* P_M f_m \|_1 \Big )}{\| \Phi_1^* f_1 \|_1 + \dots + \|\Phi_N^*f_N\|_1}, 
\end{equation*} 

\begin{equation*} 
\kappa_N^{sep}=  \sup\limits_{\sum\limits_{n=1}^N f_N \in \mathcal{H}_M} \frac{\| \mathds{1}_{\Lambda_1} \Phi_i^* f_i \|_1 + \sum\limits_{m \neq i} \|\mathds{1}_{\Lambda_m} \Phi_m^* P_Kf_m \|_1}{\| \Phi_1^* f_1 \|_1 + \dots + \|\Phi_N^*f_N\|_1},
\end{equation*}

\begin{equation*} 
\kappa_N^{inp}=  \sup_{(f_1, f_2, \dots, f_N)} \frac{ \sum\limits_{m \neq i} \|\mathds{1}_{\Lambda_N}\Phi_N^* P_M f_N \|_1}{\| \Phi_1^* f_1 \|_1 + \dots + \|\Phi_N^*f_N\|_1}.
\end{equation*} 
And similar modification if there are more components without missing parts.
\end{Rem}
\subsection{Cluster coherence analysis}
\alex{Although} Proposition \ref{Pr1} ensures the success of Algorithm \ref{AL1-3007}\leave{but}, it is not obvious and may be difficult to verify when the bound is achieved. This raises a need to create a new tool to analyze the bound in practice. For this purpose, the notion of cluster coherence was first introduced in \cite{12} to replace the notion \alex{of} joint concentration, it was then followed by many authors \cite{7,14,15,18}. There, the cluster was defined by
$$ \mu_c(\Lambda, \Phi_1; \Phi_2):=  \max_{j\in J}  \sum_{i \in \Lambda} | \langle  \phi_{1i}, \phi_{2j} \rangle |.$$
For our analysis, we optimize the definition in \cite{12} by putting  the sum inside the inner product to make the cluster coherence smaller which therefore offers more advantages in proving the success of the proposed algorithm. We also extend it to the case of general frames instead of Parseval frames.

\begin{definition}
Given two frames $\Phi_1 = (\Phi_{1i})_{i \in I}  $ and $\Phi_2 = (\Phi_{2j})_{j \in  J} $. Then the cluster coherence $\mu_c(\Lambda, \Phi_1; \Phi_2) $ of $\Phi_1$ and $\Phi_2$ with respect to the index set $\Lambda \subset I $ is defined by
\begin{equation} \label{CT300}
\mu_c(\Lambda, \Phi_1; \Phi_2):=  \max_{j\in J}  | \langle \sum_{i \in \Lambda} \phi_{1i}, \phi_{2j} \rangle |
\end{equation}
and the cluster coherence associated with missing subspace
\begin{equation} \label{CT4}
\mu_c(\Lambda,P_M \Phi_1; P_M \Phi_2):=  \max_{j\in J}  | \langle \sum_{i \in \Lambda} P_M \phi_{1i},P_M \phi_{2j} \rangle |.
\end{equation}
\end{definition}

Obviously, by triangle inequality we have 
\[ \max_{j \in J}  \Big | \langle \sum_{i \in \Lambda} \phi_{1i}, \phi_{2j} \rangle \Big | \leq   \max_{j\in J}  \sum_{i \in \Lambda} | \langle  \phi_{1i}, \phi_{2j} \rangle |, \]
and
\[ \max_{j \in J}  \Big | \langle \sum_{i \in \Lambda} P_M\phi_{1i}, P_M\phi_{2j} \rangle \Big | \leq   \max_{j\in J}  \sum_{i \in \Lambda} | \langle  P_M\phi_{1i}, P_M\phi_{2j} \rangle |. \]

To derive a new theoretical guarantee that is useful in practice, our goal is to bound the joint concentration from above by the cluster coherence. For this, we use the following lemma.

\begin{lemma} \label{Lm1} Consider $N$  frames $\Phi_1, \Phi_2, \dots, \Phi_N$ associated with index sets $\Lambda_1, \Lambda_2,$ $ \dots, \Lambda_N$, respectively. The followings hold
 \begin{enumerate}[label=(\roman*)]
     \item  
     \begin{eqnarray*}
     \kappa_N^{sep} &\leq&  \max_{1 \leq m \leq N} \inf_{\Phi_m^d \in \mathbb{D}_{\Phi_m}}  \sum_{\substack{n \neq m \\ 1 \leq n \leq N }} \mu_c(\Lambda_n,P_K \Phi_n; \Phi_m^d) \leq \mu_{c,N}^{sep}, 
        \end{eqnarray*}
   where $$ \mu_{c,N}^{sep}:= \max_{1 \leq m \leq N} \inf_{\Phi_m^d \in \mathbb{D}_{\Phi_m}}  \sum_{\substack{n \neq m \\ 1 \leq n \leq N }} \Big ( \mu_c(\Lambda_n, \Phi_n; \Phi_m^d) + \mu_c(\Lambda_n,P_M \Phi_n; \Phi_m^d) \Big ). $$
  
     \item  $$\kappa_N^{inp} \leq \mu_{c,N}^{inp}:= \max_{1 \leq m \leq N} \inf_{\Phi_m^d \in \mathbb{D}_{\Phi_m}}  \mu_c(\Lambda_m, P_M \Phi_m; \Phi_m^d ).$$
     \item 
     \begin{eqnarray*}
     \kappa_N &\leq& \mu_{c,N}:=\max_{1 \leq m \leq N} \Big ( \mu_c(\Lambda_m, P_M \Phi_m; \Phi_m^d )+ \sum_{\substack{n \neq m \\ 1 \leq n \leq N }} \mu_c(\Lambda_n, P_K \Phi_n; \Phi_m^d)   \Big) \\
     &\leq & \mu_{c,N}^{sep} + \mu_{c,N}^{inp}.
     \end{eqnarray*}
 \end{enumerate}
\end{lemma}
\begin{proof}
(i) First, we set $\alpha_{m}:=\Phi_m^*f_m,$ and $\alpha_{m}=(\alpha_{m1}, \alpha_{m2}, \cdots, ) \in l_2, \forall m \in \{ 1, 2, \dots N \}.$  This leads to $f_m = \Phi_m^d \Phi_m^* f_m = \Phi_m^d \alpha_{m}.$ Since $\sum_{m=1}^Nf_m \in \mathcal{H}_M$ yields $\sum_{m=1}^NP_K f_n=0.$
Thus, we have
\begin{eqnarray*}
 \| \mathds{1}_{\Lambda_m} \Phi_m^*P_K f_m\|_1  & = & \| \sum_{n \neq m}\mathds{1}_{\Lambda_m} \Phi_m^* P_K f_n \|_1 = \| \sum_{n \neq m}\mathds{1}_{\Lambda_m} \Phi_m^* P_K \Phi_n^d \alpha_{n}\|_1\\
& \leq & \sum_{n \neq m} \sum_{i \in \Lambda_m} \Big( \sum_j | \langle P_K\Phi_{mi}, \Phi_{nj}^d \rangle | |\alpha_{nj}| \Big) \\
&  = &  \sum_{n \neq m}\sum_{j} \Big( \sum_{i \in \Lambda_m} | \langle P_K\Phi_{mi}, \Phi_{nj}^d \rangle | \Big) |\alpha_{nj}|  \\
& \leq &  \sum_{n \neq m}\mu_c(\Lambda_m, \Phi_m; \Phi_n^d)  \|\alpha_n \|_1 
 =   \sum_{n \neq m}\mu_c(\Lambda_m, \Phi_m; \Phi_n^d)  \|\Phi_n^* f_n \|_1.
\end{eqnarray*}
Thus,
\begin{eqnarray*}
\sum_{m=1}^N\| \mathds{1}_{\Lambda_m} \Phi_m^*P_K f_m\|_1 
& \leq & \max_{1 \leq m \leq N} \Big \{  \sum_{n \neq m}\mu_c(\Lambda_m, \Phi_m; \Phi_n^d)  \Big \}  \cdot \sum_{m=1}^N \|\Phi_m^* f_m\|_1.
\end{eqnarray*}
This finishes the proof.

(ii) We set $\alpha_{m}:=\Phi_m^*f_m$ for each $m \in \{1,2,\dots, N\} $. By dual frame property, we have  $f_m = \Phi_m^d \Phi_m^* f_m = \Phi_m^d \alpha_m, \; m=1,2 \dots, N$. Thus, we obtain
\begin{eqnarray}
\|\mathds{1}_{\Lambda_m} \Phi_m^* P_Mf_m \|_1 &=& \|\mathds{1}_{\Lambda_m} \Phi_m^* P_M \Phi_m^d \Phi_m^* f_m \|_1 
= \|\mathds{1}_{\Lambda_m} (P_M\Phi_m)^* \Phi_m^d  \alpha_m \|_1  \nonumber \\
& \leq & \sum_{i \in \Lambda_m} \Big( \sum_j | \langle P_M\Phi_{mi}, \Phi_{mj}^d \rangle | |(\alpha_{m})_j| \Big)   \nonumber \\
& =& \sum_{j} \Big( \sum_{i \in \Lambda_m}| \langle P_M\Phi_{mi}, \Phi_{mj}^d \rangle | \Big) |(\alpha_{m})_j|   \nonumber \\
& \leq & \mu_c(\Lambda_m, P_M \Phi_m; \Phi_m^d ) \|\alpha_m \|_1  
=  \mu_c(\Lambda_m, P_M \Phi_m; \Phi_m^d ) \|\Phi_m^*f_m \|_1, \nonumber
\end{eqnarray}
where $\Phi_m = \{ \Phi_{mj} \}_{j \in I_m}, \alpha_n=(\alpha_{n1}, \alpha_{n2}, \dots...) \in l_2.$ 
Thus, we have
\begin{eqnarray*}
\sum_{1 \leq m \leq N}\|\mathds{1}_{\Lambda_m} \Phi_m^* P_Mf_m \|_1 \leq \max_{1 \leq m \leq N} \mu_c(\Lambda_m, P_M \Phi_m; \Phi_m^d ) \cdot \sum_{1 \leq m \leq N}  \|\Phi_m^*f_m \|_1.
\end{eqnarray*}
This concludes
$$ \kappa_N^{inp} \leq \max_{1\leq m\leq N} \mu_c(\Lambda_m, P_M \Phi_m; \Phi_m^d ).$$
\end{proof}

(iii) Similar approach to (i) and (ii) we obtain
$$ \kappa_N \leq \max_{1 \leq m \leq N} \Big (\mu_c(\Lambda_m, P_M \Phi_m; \Phi_m^d ) + \sum_{\substack{n \neq m \\ 1 \leq n \leq N }} \mu_c(\Lambda_n, P_K \Phi_n; \Phi_m^d) \Big) .$$

We now present our first theoretical guarantee for the success of Algorithm \ref{AL1-3007} based on \alex{the} notion of cluster coherence.
\begin{Th} \label{2006-3}
Let $\Phi_1, \Phi_2, \dots, \Phi_N$ be a set of frames with pairs of frame bounds $(A_1,B_1),$ $ (A_2, B_2), \cdots, (A_N, B_N)$, respectively. For $  \delta_1, \delta_2, \dots, \delta_N \allowbreak > 0,$ we fix $\delta= \sum_{m=1}^N\delta_m,$ and suppose that $f \in \mathcal{H}$ can be decomposed as $f= \sum_{m=1}^N f_m$ so that each component $f_1, f_2, \dots, f_N$ is $\delta_1, \delta_2, \dots, \delta_N-$relatively sparse in $\Phi_1, \Phi_2, \dots, \Phi_N$, respectively. Let $(f_1^\star, f_2^\star, \dots, f_N^\star)$ solve Algorithm \ref{AL1-3007}. If we have $ \mu_{c, N} < \frac{1}{2}$, then  
\begin{equation} 
\sum_{m=1}^N \| f_m^\star - f_m \|_2 \leq \frac{2\delta}{1 - 2 \mu_{c, N}}, \label{EQ90-1}
\end{equation}
where $\mu_{c,N} \leq \mu_{c,N}^{inp} + \mu_{c,N}^{sep}$ defined in Lemma \ref{Lm1}.

\end{Th}
\begin{proof}
The claim holds by  Proposition \ref{Pr1} and Lemma \ref{Lm1} (iii).
\end{proof}

Let us now discuss, in particular,  some special cases of the theoretical guarantee.
If $\Phi_1, \Phi_2, \cdots, \Phi_N$ form Parseval frames we obtain 
\begin{eqnarray}
      \mu_{c,N} \leq \max_{1 \leq m \leq N} \Big (\mu_c(\Lambda_m, P_M \Phi_m; \Phi_m) +  \sum_{\substack{n \neq m \\ 1 \leq n \leq N }} \mu_c(\Lambda_n, P_K \Phi_n; \Phi_m) \Big ) \nonumber 
\end{eqnarray} 
\begin{equation}  \label{1308-1}
    \leq \max_{1 \leq m \leq N} \Big(\mu_c(\Lambda_m, P_M \Phi_m; \Phi_m) +  \sum_{\substack{n \neq m \\ 1 \leq n \leq N }} \mu_c(\Lambda_n, P_M \Phi_n; \Phi_m) + \mu_c(\Lambda_n, \Phi_n; \Phi_m) \Big ). 
\end{equation}
Here we exploit the fact that $\Phi_m^d \in \mathbb{D}_{\Phi_m}, \; 1\leq m \leq N.$
Now in case $m=1$ we have $\mu_{c,1} \leq \mu_c(\Lambda_1, P_M \Phi_1; \Phi_1)$. Thus,
choosing \alex{a} correct set of significant coefficients which provides small cluster sparsity and cluster coherence
 guarantees the success of inpainting a single component image which was studied in \cite{7,15,18}. In addition, in case $m=2, P_M=0$ both small $\delta, \mu_{c,2}$ guarantee the success of two-component image separation as studied in \cite{11,12}, or $m=2$ for the simultaneous separation and inpainting \cite{14}. To \leave{have}\alex{gain a} better understanding, we
present the following remark. 
\begin{Rem} \label{3007-21} By  \eqref{EQ90-1}, we say that  we can separate $f_m$ from $f_n$ if  the cluster coherence $\inf\limits_{\Phi_m^d \in \mathbb{D}_{\Phi_m}, \Phi_n^d \in \mathbb{D}_{\Phi_n}} \max \{ \mu_c(\Lambda_m, \Phi_m; \Phi_n^d), \mu_c(\Lambda_n, \Phi_n; \Phi_m^d) \}$ and $\delta_m + \delta_n $   are sufficiently small and we can inpaint the component $f_n$ if both  $\mu_c(\Lambda_n, P_M\Phi_n; \Phi_n^d)$ and $\delta_n$ are sufficiently small. In that sense, a few remarks are made as follows. 
\begin{enumerate}
 \item We can separate each pair of $f_m$ and $f_n,\;  1\leq m \neq n \leq N, $ if and only if we can separate all these $N$ components from $f= f_1 + f_2 + \dots + f_N$.  
 \item We can inpaint each single component$f_n,\;  1 \leq n \leq N, $ if and only if we can  inpaint  $f= f_1 + f_2 + \dots + f_N$. 
   \item We can separate $f_m$ from $f_n$ and inpaint each of them if and only if we can simultaneously separate and inpaint these two components. 
  \item We can simultaneously separate and inpaint $N$ components if and only if we can simultaneously separate and inpaint each pair of $N$ components.  
   \item 
  For the case $f_i$ has no missing part as referred in Remark \ref{3007-20}, the cluster in \eqref{EQ90-1} is modified with the replacement of $P_M \Phi_i, P_K \Phi_i$ by $0$ and $\Phi_i,$ respectively.
\end{enumerate}
\end{Rem}
\section{$l_1$ optimization of image denoising} \label{0208-6}
 In practice, the original signal $f$ often contains noise, this therefore requires
an adaption of Algorithm \ref{AL1-3007}. For constrained optimization approach, we assume that we know $ P_Kf + \eta,$ where noise $\eta \in \mathcal{H}_K$ such that there exists $m_0$, $\| \Phi_{m_0}^* \eta \|_1 < \epsilon,$ for small $\epsilon.$ Our goal now is  to reconstruct $f$. In this situation, we can modify Algorithm \ref{AL1-3007} by
\begin{eqnarray}
 (f_1^{\star}, f_2^{\star},\dots, f_N^\star ) &= & \argmin_{(f_1, \dots, f_N)} \| \Phi_1^* f_1 \|_1 + \| \Phi_2^* f_2 \|_1 + \dots \| \Phi_N^* f_N \|_1, \nonumber \\
 &&   \text{subsect to} \quad P_K(f_1 + f_2 + \dots + f_N)= P_Kf + \eta. \label{EQ89}
  \end{eqnarray} 
For this, we can easily verify the following bound similarly as in \eqref{EQ90-1} by modifying $z_{m_0}=f_{m_0} - f_{m_0} - \eta,$ 
\begin{equation} 
\sum_{m=1}^N \| f_m^\star - f_m \|_2 \leq \frac{2\delta + 2\kappa_N\epsilon}{1 - 2 \kappa_N} \leq \frac{2\delta + 2\mu_{c, N}\epsilon}{1 - 2 \mu_{c, N}}. \label{EQ91}
\end{equation}

Another idea to solve this problem is to use a regularization term. In this approach, we propose the following algorithm to solve an unconstrained minimization problem using a regularizer $\mathcal{R}(x) >0, \forall x \in \mathcal{H}, \lambda=const >0$.
Here the constant $\lambda$  controls the trade-off between the  $l_1$ norm minimization of analysis coefficients and the prior information. We will later prove that separation, inpainting, and denoising can be all included as one unified task by the following algorithm.
\vskip 1mm
\begin{algorithm}[H]
\SetAlgoLined
\KwData{observed image $f$, a constant $\lambda >0,$ a regularizer $\mathcal{R}$, $N$ Parseval frames $\{ \Phi_1 \}_{i \in I_1}, \dots, \{ \Phi_N \}_{j \in I_N}$. }

\textbf{Compute:} $(f_1^\star, f_2^\star, \dots, f_N^\star)$, where 
\begin{eqnarray*}
    (f_1^{\star}, f_2^{\star},\dots, f_N^\star ) &= & \argmin_{(f_1, \dots, f_N)} \| \Phi_1^* f_1 \|_1 + \| \Phi_2^* f_2 \|_1 + \dots \| \Phi_N^* f_N \|_1 + \\
    && \lambda \mathcal{R}(P_K(f_1 + f_2 + \dots + f_N) - P_Kf ).
   \end{eqnarray*}  \vskip 0.5mm
\KwResult{ recovered components $f_1^{\star}, f_2^{\star}, \dots, f_N^\star.$ } 
\vskip 1mm
 \caption{\label{AL2-3007} Unconstrained optimization}
\end{algorithm}
\vskip 1mm

For a special case of our algorithm, one  without theoretical analysis proposed in \cite{20} which used an $L_2$ regularization term
 \begin{equation*}
    (f_1^{\star}, f_2^{\star}) =  \argmin_{f_1, f_2} \| \Phi_1^* f_1 \|_1 + \| \Phi_2^* f_2 \|_1 + \lambda \| f - f_1 - f_2 \|_2^2. 
   \end{equation*}
This empirical work using Basic pursuit exhibited very good results for separating pointlike and curvelike parts. 
By utilizing the notions of cluster coherence, we also have a theoretical guarantee for the success of Algorithm \ref{AL2-3007} by the following result.
\begin{Th} \label{2006-4}
Let $\Phi_1, \Phi_2, \dots, \Phi_N$ be a set frames with frame bounds $(A_1,B_1),$ $ (A_2, B_2), \cdots, (A_N, B_N)$, respectively. For $  \delta_1, \delta_2, \dots, \delta_N \allowbreak > 0,$ fix $\delta= \sum_{m=1}^N\delta_m,$ and suppose that $f \in \mathcal{H}$ can be decomposed as $f= \sum_{m=1}^N f_m$ so that each component $f_1, f_2, \dots, f_N$ is $\delta_1, \delta_2, \dots, \delta_N-$relatively sparse in $\Phi_1, \Phi_2, \dots, \Phi_N$, respectively. We assume that $\Lambda_1, \Lambda_2, \dots, \Lambda_N$ satisfy
\begin{equation} \label{EQ1}
    \sum_{1 \leq m \leq N} \| \mathds{1}_{\Lambda_m} \Phi_m z \|_1 \leq \lambda \mathcal{R}(z), \;  \forall z \in \mathcal{H}, 
\end{equation} 
and $\mu_{c, N}< \frac{1}{2}$. 
 Let $(f_1^\star, f_2^\star, \dots, f_N^\star)$ solve Algorithm \ref{AL2-3007}, then we have 
 \begin{equation} \label{EQ76}
\sum_{m=1}^N \| f_m^\star - f_m \|_2 \leq \frac{2\delta}{1 - 2 \mu_{c, N}}.
\end{equation}
\end{Th}
\begin{proof}
For $1 \leq m,n \leq N$, we set $z_m= f^\star_m - f_m, z=\sum_{m=1}^Nz_m$ and $\alpha_{m}:=\Phi_m^*z_m$ . For $\Phi_m^d \in \mathbb{D}_{\Phi},$ we have
\begin{equation}
    \Phi_m^d \alpha_{m} = \Phi_m^d \Phi_m^*z_m =z_m, \label{B0}
\end{equation}
and
\begin{eqnarray} 
\sum_{m=1}^N \|  f_m^\star - f_m \|_2 &=&  \sum_{m=1}^N  \| \Phi_m^*(f^\star_m - f_m ) \|_2 \nonumber \\
&\leq &  \sum_{i=1}^N \| \Phi_m^*(z_m) \|_1  : = \mathbb{T}.  
\end{eqnarray}
For each $m \in \{ 1, 2, \dots, N \} $, we observe that
\begin{eqnarray}
\|\mathds{1}_{\Lambda_m} \Phi_m^*z_m \|_1  &\leq & \|\mathds{1}_{\Lambda_m} \Phi_m^*P_Kz_m \|_1  +  \|\mathds{1}_{\Lambda_m} \Phi_m^* P_Mz_m \|_1  \nonumber \\
& \leq &  \sum_{n \neq m} \|\mathds{1}_{\Lambda_m} \Phi_m^*P_Kz_n \|_1 +  \|\mathds{1}_{\Lambda_m} \Phi_m^*P_Kz \|_1  + \|\mathds{1}_{\Lambda_m} \Phi_m^*P_Mz_m \|_1. \nonumber \\ \label{B2}
\end{eqnarray}
By \eqref{B0}, we have 
\begin{eqnarray}
\|\mathds{1}_{\Lambda_m} \Phi_m^*P_Kz_n \|_1 & = & \|\mathds{1}_{\Lambda_m} \Phi_m^*P_K \Phi_n^d \alpha_{n} \|_1  
 =  \|\mathds{1}_{\Lambda_m} (P_K\Phi_m)^* \Phi_n^d \alpha_{n} \|_1 \nonumber  \\
& \leq & \sum_{i \in \Lambda_m} \Big( \sum_j | \langle P_K\Phi_{mi},  \Phi_{nj}^d \rangle | |(\alpha_{n})_j| \Big)   \nonumber \\
& =& \sum_{j} \Big( \sum_{i \in \Lambda_m}| \langle P_K\Phi_{mi}, \Phi_{nj}^d \rangle | \Big) |(\alpha_{n})_j|   \nonumber \\
& \leq & \mu_c(\Lambda_m, P_K \Phi_m; \Phi_n^d ) \|\alpha_n \|_1 
=  \mu_c(\Lambda_m, P_K \Phi_m; \Phi_n^d ) \|\Phi_n^*z_n \|_1. \nonumber \\ \label{B3}
\end{eqnarray}

Similarly, we have
\begin{equation}
    \|\mathds{1}_{\Lambda_m} \Phi_m^*P_Mz_m \|_1 \leq  \mu_c(\Lambda_m, P_M \Phi_m; \Phi_m^d ) \|\Phi_m^* z_m \|_1. \label{B4}
\end{equation}
Combining \eqref{B2}, \eqref{B3} and \eqref{B4}, we obtain
\begin{eqnarray}
\sum_{m=1}^N \|\mathds{1}_{\Lambda_m} \Phi_m^*z_m \|_1 &\leq & \sum_{m=1}^N \Big ( \sum_{n \neq m} \mu_c(\Lambda_m, P_K \Phi_m; \Phi_n^d ) \| \Phi_n^* z_n \|_1 +  \nonumber \\
&&  
   \mu_c(\Lambda_m, P_M \Phi_m; \Phi_m^d) \| \Phi_m^* z_m \|_1 \Big )+ \sum_{m=1}^N \|\mathds{1}_{\Lambda_m} \Phi_m^*P_Kz \|_1 \nonumber \\
& \leq & \max_{n} \Big ( \mu_c(\Lambda_m, P_M \Phi_m; \Phi_m^d )  + \sum_{m \neq n}\mu_c(\Lambda_m, P_K \Phi_m; \Phi_n^d ) \Big ) \mathbb{T} + \nonumber \\
&&  \sum_{m=1}^N \|\mathds{1}_{\Lambda_m} \Phi_m^*P_Kz \|_1.\label{B5}
\end{eqnarray}

By triangle inequality, we have
\begin{eqnarray}
\mathbb{T} &= & \sum_{m=1}^N\| \mathds{1}_{\Lambda_m} \Phi_m^* z_m \|_1 +\sum_{m=1}^N \| \mathds{1}_{\Lambda_m^c} \Phi_m^* (f^\star_m - f_m) \|_1 \nonumber \\
& \leq & \sum_{m=1}^N\| \mathds{1}_{\Lambda_m} \Phi_m^* z_m \|_1 +\sum_{m=1}^N \| \mathds{1}_{\Lambda_m^c} \Phi_m^* f^\star_m  \|_1 + \delta \nonumber \\
& = & \sum_{m=1}^N\| \mathds{1}_{\Lambda_m} \Phi_m^* z_m \|_1 +\sum_{m=1}^N \| \Phi_m^* f^\star_m  \|_1 - \sum_{m=1}^N \|\mathds{1}_{\Lambda_m} \Phi_m^* f^\star_m  \|_1 + \delta  \nonumber \\
& \leq & \sum_{m=1}^N\| \mathds{1}_{\Lambda_m} \Phi_m^* z_m \|_1 +\sum_{m=1}^N \| \Phi_m^* f^\star_m  \|_1 +  \nonumber \\
&&
\sum_{m=1}^N \Big (\|\mathds{1}_{\Lambda_m} \Phi_m^* z_m  \|_1 -\|\mathds{1}_{\Lambda_m} \Phi_m^* f_m  \|_1 \Big ) + \delta. 
\label{B6} 
\end{eqnarray}
We note that $(f_1^\star, f^\star_2, \dots, f_N^\star) $ is a minimizer of $\textup{Algorithm \ref{AL1-3007}}$. Thus
\begin{equation}
\sum_{m=1}^N \| \Phi_m^* f^\star_m \|_1 \leq  \sum_{m=1}^N \| \Phi_m^* f_m \|_1 -\lambda \|P_Kz\|_1. \label{B7}
\end{equation}
By  \eqref{B5}, \eqref{B6}, \eqref{B7}, and \eqref{EQ1}, we have
\begin{eqnarray*}
\mathbb{T} & \leq  & 2 \sum_{m=1}^N\| \mathds{1}_{\Lambda_m} \Phi_m^* z_m \|_1 - \lambda \|P_Kz \|_1 + 2 \delta \nonumber \\
& \leq & 2\max_{n} \Big ( \mu_c(\Lambda_m, P_M \Phi_m; \Phi_m^d )  + \sum_{m \neq n}\mu_c(\Lambda_m, P_K \Phi_m; \Phi_n^{d}) \Big ) \mathbb{T} + \nonumber \\
&&  \Big ( \sum_{m=1}^N \|\mathds{1}_{\Lambda_m} \Phi_m^*P_Kz \|_1 - \lambda \| P_Kz \|_1 \Big ) + 2 \delta \nonumber \\
& \leq & 2\max_{n} \Big ( \mu_c(\Lambda_m, P_M \Phi_m; \Phi_m^d )  + \sum_{m \neq n}\mu_c(\Lambda_m, P_K \Phi_m; \Phi_n^d ) \Big ) \mathbb{T} + 2 \delta. \nonumber
\end{eqnarray*} 
Thus, we finally obtain $$\mathbb{T} \leq \frac{2 \delta}{1 - 2  \mu_{c, N} }.$$
This concludes the claim.
\end{proof}

In our analysis, we will later provide an analysis of Algorithm \ref{AL2-3007} with $\mathcal{R}(x) = \|x\|_1$ for separating points from curves and texture while inpainting the missing parts of curves and texture in Section \ref{0208-10}. For this, we have the following corollary.
\begin{Cor} \label{CR1}
Let $\Phi_1, \Phi_2, \dots, \Phi_N$ be a set of Parseval frames. For $  \delta_1, \delta_2, \dots, \delta_N \allowbreak > 0,$ we fix $\delta= \sum_{m=1}^N\delta_m,$ and suppose that $f \in L^2(\mathbb{R}^2)$ can be decomposed as $x= \sum_{m=1}^N f_m$ so that each component $f_1, f_2, \dots, f_N$ is $\delta_1, \delta_2, \dots, \delta_N-$relatively sparse in $\Phi_1, \Phi_2, \dots, \Phi_N$, respectively. Let $(f_1^\star, f_2^\star, \dots, f_N^\star)$ solve Algorithm \ref{AL2-3007} with $\mathcal{R}(x) = \|x\|_1$, and we assume that $\Lambda_1, \Lambda_2, \dots, \Lambda_N$ satisfy
\begin{equation} \label{2006-5}
    \sum_{1 \leq m \leq N} \| \mathds{1}_{\Lambda_m} \Phi_m z \|_1 \leq \lambda \| z \|_1, \;  \forall z \in L^2(\mathbb{R}^2).
\end{equation} 
If  $\mu_{c,N}< \frac{1}{2}$, then we have 
 \begin{equation} \label{EQ760}
\sum_{m=1}^N \| f_m^\star - f_m \|_2 \leq \frac{2\delta}{1 - 2 \mu_{c,N}}.
\end{equation}
\end{Cor}

\section{Inpainting and denoising of multiple component images in applications} \label{0208-10}

In image processing, the performance of existing algorithms is effective in different scenarios depending on the geometry of  underlying components. Natural images may contain more than two geometric structures, for instance, astronomical images are often composed of points, curves, and texture. One interesting question in application is how to inpaint an image which is composed of multiple components the known part of the  groundtruth image $f \in \mathcal{H}=L^2(\mathbb{R}^2).$ 
Mathematically, we assume that we have the following decomposition
\begin{equation}
    f=\mathcal{P}+\mathcal{C} + \mathcal{T}, 
\end{equation}
where $f$ is the signal of our interest, $\mathcal{P}, \mathcal{C},$ and $ \mathcal{T} $ are three underlying components that are unknown to us. In addition, we assume that $\mathcal{H}$ can be decomposed into the known and unknown subspace by $\mathcal{H} = \mathcal{H}_K \oplus \mathcal{H}_M$. Our goal is to recover $\mathcal{P}, \mathcal{C}$ and $\mathcal{T}$ with only the known part $P_kf$.
\subsection{Model of components}
In this subsection, we introduce models of pointlike, curvelike, and  texture structures which we aim to recover. In our analysis, the model of texture is motivated  from \cite{14} with $s=1.$ We modified the model of texture for our need as follows.
\begin{definition} 
Let $g \in L^2(\mathbb{R}^2)$ be a window with $\hat{g} \in C^{\infty}(\mathbb{R}^2),$ and frequency support $\supp \hat{g} \in [-1,1]^2$ satisfying the partition of unity condition
\begin{equation} 
     \sum_{n \in \mathbb{Z}^2}|\hat{g}(\xi +n)|^2=1, \quad \xi \in \mathbb{R}^2. 
\end{equation}
 Let $I_T \subseteq \mathbb{Z}^2$ be a subset of Fourier elements. A  texture is defined by
\begin{equation} 
\mathcal{T}(x) = \sum_{n \in I_T} d_{n} g(x) e^{2 \pi i x^\top n },
\end{equation}
where $(d_{n})_{n \in \mathbb{Z}^2}$ denotes  a bounded sequence of complex numbers.
\end{definition}

Point-like structure $\mathcal{P}$ is modeled as 
\begin{equation} \label{EQ60}
    \mathcal{P}(x)= \sum_{1}^K | x- x_i |^{-3/2}.
\end{equation}
Since it is certainly impossible to reconstruct  point singularities in the missing part, so we assume that $x_i \in \mathcal{H}_K, \forall i=1,2,\dots, K.$

Curvelike structure is modeled as a line distribution $w \mathcal{L}$ acting on \emph{Schwartz functions} by
\begin{equation} \label{EQ61}
\langle w \mathcal{L} , f \rangle = \int_{-\rho}^{\rho}  w(x_1) f(x_1, 0) dx_1, \; f \in \mathcal{S}(\mathbb{R}^2).
\end{equation} 
 
\subsection{Sparse representations}
In this section, we focus on tools of harmonic analysis for sparsely approximate each component. In particular, we use the following sparsifying systems which might be best adapted to the geometry of its singularities. 
\begin{itemize}
 \item Gabor tight frames: A tight frame with time-frequency balanced elements.
    \item Radial wavelets: Bandlimited wavelets which form a Parseval frame of isotropic generating elements.
    \item Shearlets: A highly directional tight frame with increasingly anisotropic elements at fine scales.
   
\end{itemize}
We now fix a constant $\epsilon$ satisfying
\begin{equation}
    0< \epsilon < \frac{1}{4}. \label{0907-1}
\end{equation}
We consider the  following Gabor tight frame $\mathbf{G}= \{ g_{m,n}(x) \}_{m,n \in \mathbb{Z}^2}$  defined in the frequency domain as
\begin{equation} \label{EQ0}
   \hat{g}_{m,n}(\xi) = \hat{g}(\xi - n) e^{2 \pi i  \xi^\top \frac{m}{2}}.  
\end{equation}
This system constitutes a tight frame for $L^2(\mathbb{R} ^2)$, see \cite{1} for more details.
Different from \cite{14} where two conditions on th sparse set of indexes $I_T$ were used, we assume that  
\begin{equation} \label{1907}
    | I_T^{\pm} \cap \mathcal{A}_{j} | \leq 2^{(1-\epsilon)j}  \quad \textup{as} \, j \rightarrow \infty . 
\end{equation} 
where 
   $ I_T^{\pm} = \{n^\prime \in \mathbb{Z}^2 \mid \exists n \in I_T: | n^\prime -n| \leq 1  \}.$ 
   Roughly speaking, we assume that at every scale $j$,  the number of non-zero Gabor elements with the same position, generating $\mathcal{T}_{j}$, is  not too large. We refer to Fig \ref{Fig1407} for an illustration.
\begin{figure}[H] 
 \centering
 \includegraphics[width=0.4\textwidth]{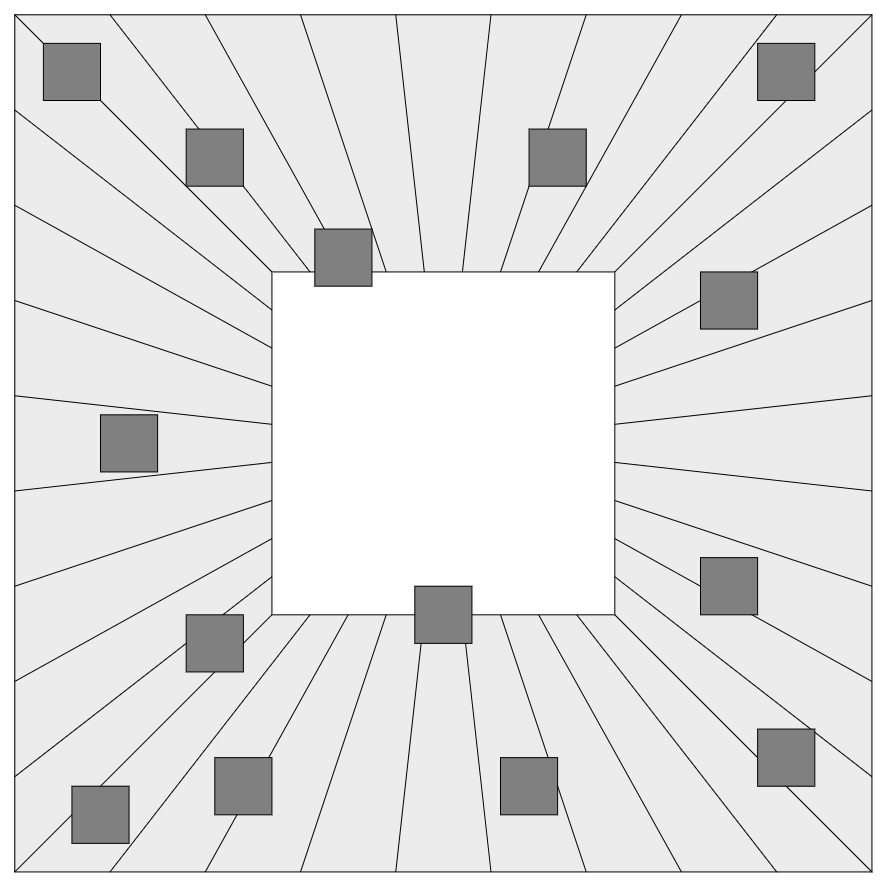}
\;
  \caption{\label{Fig1407} Interaction between Gabor elements (black) with wavelet and shearlet elements (gray)} 
 \end{figure}

For the construction of wavelets and shearlets, we choose smooth Parseval frames of shearlets \cite{19} and  a Parseval frame of wavelets which share the same window function.

Let $\Xi$ be a function in $\mathcal{S}(\mathbb{R})$ satisfying $0 \leq \hat{\Xi}(u) \leq 1$ for $u\in\mathbb{R}$, $\hat{\Xi}(u) =1$ for $u \in [-1/16, 1/16]$ and $\supp \hat{\Xi} \subset [-1/8, 1/8]$. Define the \emph{low pass function} $\Theta(\xi)$ and the \emph{corona scaling functions} for $j \in \mathbb{N} $ and $\xi= (\xi_1,\xi_2 ) \in \mathbb{R}^2,$
$$ \hat{\Theta}(\xi):= \hat{\Xi}(\xi_1) \hat{\Xi}(\xi_2),$$
\begin{equation} \label{CT35}
    W(\xi):= \sqrt{\hat{\Theta}^2(2^{-2}\xi)-\hat{\Theta}^2(\xi)}, \quad W_j(\xi):= W(2^{-2j} \xi).
\end{equation} 
It is easy to see that we have the partition of unity property
\begin{equation} \label{CT3}
 \hat{\Theta}^2(\xi) + \sum_{j \geq 0 } W_j^2(\xi) = 1, \quad \xi \in \mathbb{R}^2.
\end{equation}
    The smooth Parseval frame of  wavelets $\{ \phi_{j,k} \}_{j \in \mathbb{N}_0, k \in \mathbb{Z}^2}$ is  defined by
\begin{equation}
    \hat{\phi}_{j,k}=2^{-2j}W_j(\xi) e^{2 \pi i 2^{-2j} \xi^{\top} k}.
\end{equation}
Next, we use a bump-like function $ \upsilon  \in C^\infty (\mathbb{R})$ to produce the directional scaling feature of the system. Suppose $\supp(\upsilon) \subset [-1,1] $ and $| \upsilon (u-1)|^2 + |\upsilon(u)|^2 + |\upsilon(u+1)|^2 =1$ for $ u \in [-1,1]$. Define the \emph{horizontal frequency cone} and the \emph{vertical frequency cone}
\begin{equation} \label{EQ30}
\mathcal{C}_{\rm{h}} := \Big \{ (\xi_1, \xi_2) \in \mathbb{R}^2 \mid \Big | \phantom{.}  \frac{\xi_2}{\xi_1} \Big | \leq 1 \Big \} 
\end{equation}
and
\begin{equation} \label{EQ31}
     \mathcal{C}_{\rm{v}} := \Big \{ (\xi_2, \xi_1) \in \mathbb{R}^2 \mid  \Big  | \frac{\xi_1}{\xi_2} \Big | \leq 1 \Big \}.
\end{equation}
Define  the cone functions $V_{\rm{h}}, V_{\rm{v}}$ by
 \begin{equation} \label{EQ32}
 V_{\rm{h}}(\xi): = \upsilon \Big ( \frac{\xi_2}{\xi_1} \Big ), \quad V_{\rm{v}}(\xi): = \upsilon \Big ( \frac{\xi_1}{\xi_2} \Big ). 
 \end{equation}
The \emph{shearing} and \emph{scaling matrix} are defined by
\begin{equation} \label{EQ33}
A_{\rm{h}} := 
\begin{bmatrix}
    2^2   & 0 \\
      0  & 2 \\
   \end{bmatrix}, \quad S_{\rm{h}} := \begin{bmatrix}
    1  & 1 \\
      0  & 1 \\
   \end{bmatrix} ,
\end{equation}

\begin{equation} \label{EQ34}
 A_{\rm{v}} := 
\begin{bmatrix}
    2 & 0 \\
      0  & 2^2 \\
   \end{bmatrix}, \quad S_{\rm{v}} := \begin{bmatrix}
    1  & 0 \\
      1  & 1 \\
   \end{bmatrix}.
\end{equation}
\begin{definition} \label{DE10}  The smooth Parseval of shearlets $\mathbf{\Psi}= \{ \psi_{j,l,k}^{ \rm{\iota}} \}_{j,l,k}, j \in \mathbb{N}_0, l \in \mathbb{Z}, 0 \leq |l| \leq 2^{j}, k \in \mathbb{Z}^2, \rm{\iota} \in \{ \rm{v}, \rm{h} \} $ is defined by
\begin{equation}
     \hat{\psi}_{j,l,k}^{ \iota}(\xi): = 2^{-3j/2}W_j(\xi) V_{\iota} \Big ( \xi^T A_{ \iota}^{-j} S^{-l}_{\iota} \Big ) e^{-2\pi i \xi^T A_{ \iota}^{-j} S^{-l}_{\iota}k}, \quad \xi \in \mathbb{R}^2. 
\end{equation}
\end{definition}
\begin{figure}[H] 
\centering
\includegraphics[width=220pt, height=170pt]{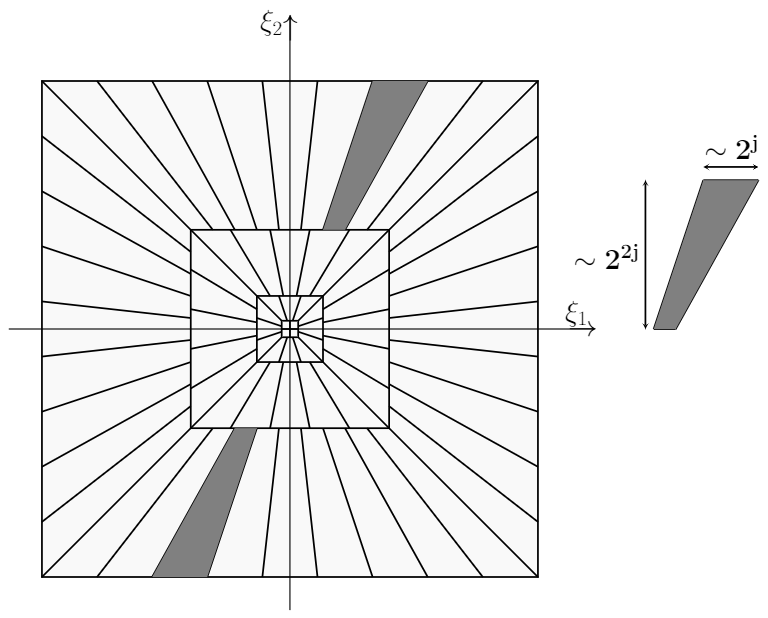}
\caption{\label{2607-1} Frequency tiling of a cone-adapted shearlet. } 
\end{figure} 
For an illustration, the tiling of the frequency domain induced by shearlets is depicted in Fig. \ref{2607-1}.

\subsection{Mutiscale approach}
In our analysis, the simultaneous separation and inpainting problem is analyzed at each scale $j$. To achieve this goal we use the following class of \emph{frequency filters} $F_j$ 
\begin{equation} \label{CT25}
   \hat{F}_j(\xi):=W_j(\xi) = W(\xi/2^{2j}), \quad \forall j \geq 0, \xi \in \mathbb{R}^2, 
\end{equation}
and in  low frequency part 
$$ \hat{F}_{low}(\xi):= \Theta(\xi), \quad \xi \in \mathbb{R}^2.$$
By using these filters, we can   decompose $f$ into small pieces  by $f_{(j)} = F_j * f, \; j \geq 0$ and $f_{low}= F_{low}*f$. Then \eqref{CT3} allows us to reconstruct $f$ by
\begin{equation} \label{0408-1}
    f  = F_{low}*f_{low}+  \sum_{j \geq 0} F_j * f_{(j)}.
\end{equation}
For multi-scale separation and inpainting, we intend to apply Algorithm \ref{AL1-3007}, Algorithm \ref{AL2-3007} for each sub-image $f_{(j)}$, then the whole image is reconstructed by $\eqref{0408-1}$. Here we assume that $f_{(j)}= \mathcal{P}_j + w\mathcal{L}_j + \mathcal{T}_j$. In the other words, instead of considering the whole signal we now consider the problem of simultaneous separation and inpainting at each scale. For the theoretical analysis, we assume that the model of missing subspace is chosen at each scale $j$ to be $\mathcal{H}_M=\mathds{1}_{\mathcal{M}_{h_j}}$ where $\mathcal{M}_{h_j}= \{ x= (x_1, x_2) \mid |x_1| \leq h_j \}$ is a horizontal strip domain. 
This model of missing part was first introduced in \cite{15} and then also adopted in \cite{7,18}. The orthogonal projection associated with the missing trip is then defined by\begin{equation} \label{1007-4}
P_{M,j} = \mathds{1}_{\{ |x_1| \leq h_j \}}.
\end{equation} For an illustration, we refer to Fig. \ref{Fig3007-30}.
We now face to the task of reconstructing the sub-image $f_{(j)}$ with the knowledge of  $P_{K,j}f_{(j)}$, where 
$P_{K,j} =1-P_{M,j}.$
\begin{figure}[H] 
 \centering
 \includegraphics[width=0.4\textwidth]{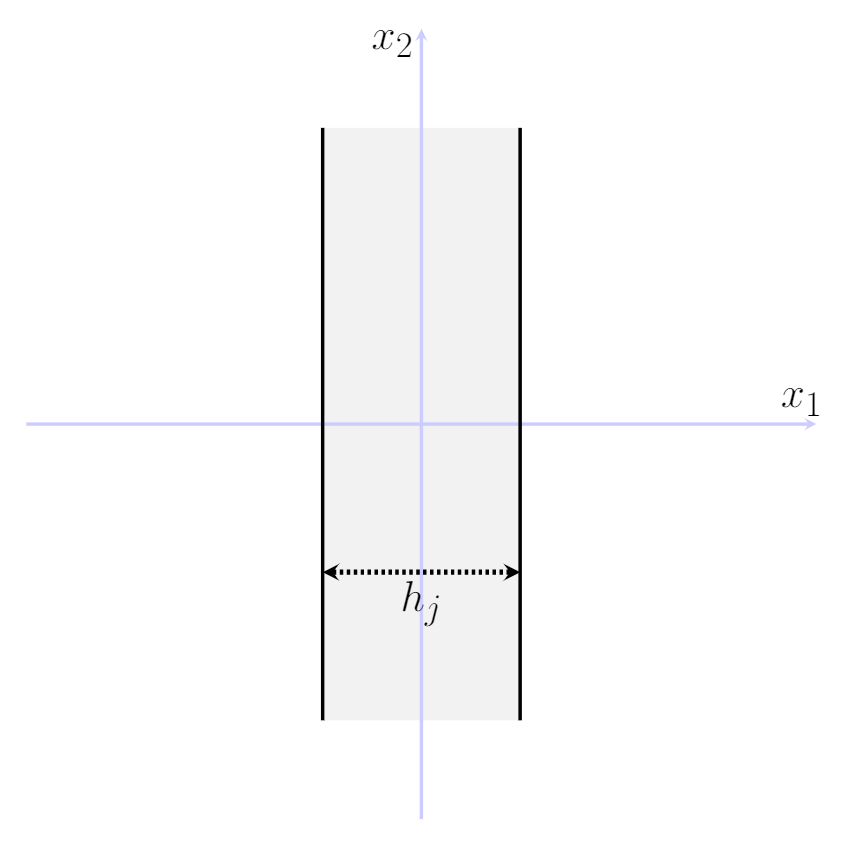}
\;
  \caption{\label{Fig3007-30}  Model of the missing trip at scale $j$.} 
 \end{figure} 
 
In addition, we assume that the energy of texture is comparable to point and curvilinear singularities, i.e., $\| \mathcal{P}_j \|_2 \approx \| w\mathcal{L}_j \|_2 \approx \| \mathcal{T}_j \|_2.$ This is due to the fact that the problem becomes trivial in the case of different levels of energy.  Actually, we can derive  similar theoretical results by our approach if the energy of each component has a lower bound at each scale. However, for the sake of simplicity, we assume that  $\| \mathcal{T}_j \|_2 \approx 2^{2j}$  since we have $\| \mathcal{P}_j \|_2 \approx \| w\mathcal{L}_j \|_2 \approx 2^{2j}.$ These energy balancing conditions make the problem of geometric separation  challenging at every scale.
 
\subsection{Main results}
Our analysis is based on the fact that the geometric property underlying each component is often encoded in the frame coefficients. In addition, important coefficients cluster geometrically in phase space by microlocal analysis. This therefore allows us to separate components based on the machinery in Theorem \ref{2006-3} and \ref{2006-4}. 
Thus, the correct choice of clusters plays a central role in our analysis.  We aim to define the optimal set of significant coefficients to derive good cluster sparsity as well as cluster coherence.  By Theorem \ref{2006-3} and \ref{2006-4}, if  clusters are chosen large we might have very a good approximation but the cluster coherence will  be big which leads to loose control of the $L_2$ errors. 

We now define clusters for wavelet, shearlet, and Gabor systems.
First the cluster of significant wavelet coefficients is defined by
\begin{equation}
   \Lambda_{1,j}^{\pm}:= \Lambda_{1, j-1} \cup \Lambda_{1,j} \cup \Lambda_{1, j+1}, \; \forall j \geq 1,
\end{equation}   
where $  \Lambda_{1,j} = \{ p \mid p \in \mathbb{Z}^2, |p| \leq 2^{\epsilon j} \}. $
For the line singularity $w\mathcal{L}_j$,  we  define the set of significant coefficients of shearlet system by 
\begin{equation} \label{CT6}
\Lambda_{2,j}^{\pm}:= \Lambda_{2,j-1} \cup \Lambda_{2,j} \cup \Lambda_{2,j+1}, \quad  \forall j \geq 2,
\end{equation}
where
$\Lambda_{2,j}: = \Big \{ (j,l,k;{\rm v}) \; | \; |l| \leq 1, k=(k_1, k_2) \in \mathbb{Z}^2, |k_2 - lk_1| \leq 2^{\epsilon j} \Big \}.$
We define the cluster for texture at scale $j$ to be
\begin{equation} \label{CT7}
\Lambda_{3,j}^\pm := \Big ( \mathbb{Z}^2 \cap B(0,M_j) \Big )  \times \Big ( I_T^{\pm} \cap \mathcal{A}_{j} \Big ),
\end{equation}
where $I_T^{\pm} = \{n^\prime \in \mathbb{Z}^2 \mid \exists n \in I_T: | n^\prime -n| \leq 1  \},$ $M_j:= 2^{\epsilon j/6}$ and $B(0,r)$ denotes the closed $l_2$ ball around the origin in $\mathbb{R}^2$.
Here we remark that  the index set $\Lambda_{1,j}^{\pm}, \Lambda_{2,j}^{\pm}, \Lambda_{,j}^{\pm}$ were chosen instead of $\Lambda_{1j}, \Lambda_{2,j}, \Lambda_{3,j}$ since  there have overlap supports between them and their neighbors.

For later use, let us now introduce some useful notations related to shearlets 
\begin{eqnarray} 
\langle |x| \rangle &:=& (1 + |x|^2)^{1/2} \nonumber \\
\Delta &:=& \Big \{(j,l,k; \rm{\iota}) \mid j \geq 0, |l| < 2^{j}, k \in \mathbb{Z}^2, \rm{\iota} \in \{ \rm{h}, \rm{v} \}  \Big \}  \nonumber \\
\Delta_j &:=& \{ (j^\prime,k,l;  \rm{\iota}) \in \Delta \mid j^\prime = j \}, \quad j \geq 0,  \nonumber\\
\Delta_j^{\pm} &:=& \Delta_{j-1} \cup \Delta_j \cup \Delta_{j+1}, \; \text{where } \Delta_{-1}=\emptyset.\nonumber
\end{eqnarray} 
By integration by parts and standard arguments, we have the decay estimates of Gabor, wavelet, and shearlet elements by the following lemma.
\begin{lemma} \label{0907-2}
For each $ N =1,2, \dots $ there exists a constant $C_N$ independent of $j$ such that the followings hold for $j \in \mathbb{N}, m,n,l,k, p \in \mathbb{Z}, |l| \leq 2^{j}$.
\begin{enumerate}[label=(\roman*)]
\item $ |g_{m,n}(x)| \leq C_N  \langle | x_1 + \frac{m_1}{2}|\rangle^{-N}\langle |x_2 + \frac{m_2}{2}|\rangle^{-N},$ 
\item $ |\phi_{j,p}(x)| \leq C_N \cdot 2^{2j} \cdot  \langle |2^{2j} x_1 + p_1|\rangle^{-N}\langle |2^{2j} x_2 + p_2|\rangle^{-N}.$
\item $ | \psi_{j, l, k}^{\rm{v}}(x)| \leq C_N \cdot 2^{3j/2} \cdot \langle |2^{j}x_1 - k_1|\rangle^{-N} \langle |2^{2j}x_2+ l2^{ j } x_1-k_2|\rangle^{-N}, $ \\
$ | \psi_{j, l, k}^{\rm{h}}(x)| \leq C_N \cdot 2^{3j/2} \cdot \langle |2^{2j}x_1+l2^{ j}x_2-k_1)|\rangle^{-N} \langle |2^{ j}x_2-k_2|\rangle^{-N}, $ 
\item  $ |\langle g_{m,n},  \phi_{j, p} \rangle|  \leq C_N \cdot 2^{-2j},$  \\
$ |\langle g_{m,n},  \psi_{j, l, k}^{ \rm{\iota}} \rangle|  \leq C_N \cdot 2^{-3j/2},$ \quad   $\forall 
 \rm{\iota \in \{ \rm{v}, \rm{h} \}},$ \\
 $ |\langle \phi_{j,p},  \psi_{j, l, k}^{ \rm{\iota}} \rangle|  \leq C_N \cdot 2^{-j/2},$ \quad   $\forall 
 \rm{\iota \in \{ \rm{v}, \rm{h} \}}.$
\end{enumerate}
\end{lemma}

We now present the following convergence result which shows the success of Algorithm $\ref{AL1-3007}$ in simultaneously  inpainting and separating three components of points, curves, and texture.

\begin{Th}  \label{3007-10}
For $0<h_j = o(2^{-(1 + \epsilon)j})$ with $\epsilon$ satisfying $0 < \epsilon <   \frac{1}{3},$
the recovery error provided by Algorithm \ref{AL1-3007} decays rapidly and we  have asymptotically perfect simultaneous separation and inpainting. Namely, for all $ N \in \mathbb{N}_0,$
\begin{equation} 
 \frac{\| \mathcal{P}_j^* - \mathcal{P}_j  \|_2}{\| \mathcal{P}_j\|_2}+ \frac{\| \mathcal{C}_j^* - w\mathcal{L}_j  \|_2}{\| w\mathcal{L}_j \|_2}+ \frac{\| \mathcal{T}_j^* - \mathcal{T}_{j}  \|_2}{ \| \mathcal{T}_{j} \|_2} = o(2^{-Nj}) \rightarrow 0, \quad j \rightarrow \infty,  \label{3007}
\end{equation}
where $(\mathcal{P}_j^*, \mathcal{L}_j^*, \mathcal{T}_j^*)$ is the solution of Algorithm $\ref{AL1-3007}$ and $(\mathcal{P}_j, \mathcal{C}_j, \mathcal{T}_{j})$ are ground truth components.  
\end{Th}
For noisy images, we have the following result.
\begin{Th}  \label{3007-11}
For $0<h_j = o(2^{-(1 + \epsilon)j})$ with $\epsilon$ satisfying $0 < \epsilon <   \frac{1}{3},$
the recovery error provided by Algorithm \ref{AL2-3007} decays rapidly and we  have asymptotically perfect simultaneous separation and inpainting,
\begin{equation} 
 \frac{\| \mathcal{P}_j^* - \mathcal{P}_j  \|_2}{\| \mathcal{P}_j\|_2}+ \frac{\| \mathcal{C}_j^* - w\mathcal{L}_j  \|_2}{\| w\mathcal{L}_j \|_2}+ \frac{\| \mathcal{T}_j^* - \mathcal{T}_{j}  \|_2}{ \| \mathcal{T}_{j} \|_2} = o(2^{-Nj}) \rightarrow 0, \quad j \rightarrow \infty,  \label{3007-100}
\end{equation}
where $(\mathcal{P}_j, \mathcal{C}_j, \mathcal{T}_{j})$ are ground truth components and $(\mathcal{P}_j^*, \mathcal{L}_j^*, \mathcal{T}_j^*)$ is the solution of Algorithm $\ref{AL2-3007}$ with $\mathcal{R}(z)=\| z \|_1, \lambda_j = 2^{2j}.$   
\end{Th}

 These two theorems show that the $L_2$ errors of the reconstruction by Algorithm \ref{AL1-3007} and  \ref{AL2-3007}
can be made arbitrarily small in the limit as the scale tends to zero.
We would like to remind the reader that by Remark \ref{3007-21} we need to separate each pair of components while inpaint each of them to  guarantee the success of our algorithms. The proof of  two main theorems relies on a number of estimates  based on the general theoretical guarantee with respect to Remark  \ref{3007-21}.

Before giving a proof for this theorem, we consider the following materials.
\subsection{Estimate of  approximation errors}
\begin{lemma} \label{3007-1}
We have $\delta_{1j} = \circ(2^{-Nj}), \quad j \rightarrow \infty.$
\end{lemma}
\begin{proof}

Without loss of generality, we can restrict the general case to  $\mathcal{P}(x)=\frac{1}{|x|^{3/2}}.$ By definition, we have
\begin{eqnarray}
\delta_{1,j} &=& \sum_{(j^{\prime},m) \notin \Lambda_{1,j}}\langle \psi_{j^{\prime},m}, \mathcal{P}_j \rangle  \nonumber \\
&=& \sum_{j^{\prime} \geq 0, |m| > 2^{\epsilon j} }\int_{\mathbb{R}^2}  2^{-2j}W_{j^{\prime}}(\xi)e^{2 \pi i\xi^T\frac{m}{2^{2j^{\prime}}}}W_j(\xi) \frac{1}{|\xi|}  d\xi. \nonumber 
\end{eqnarray}
We can assume that $j^{\prime}=j$ due to the fact that $W_{j^{\prime}}(\xi)W_j(\xi) =0$ for $|j^{\prime} - j | >1 $.
Applying the change of variables $\eta=(\eta_1, \eta_2) =2^{-2j}(\xi_1, \xi_2),$ we obtain
\begin{eqnarray}
\delta_{1,j} 
& \lesssim & 2^j \sum_{|m| > 2^{\epsilon j} }\int_{\mathbb{R}^2}  \frac{1}{|\eta|}  W^2(\eta)e^{2 \pi i\eta^Tm} d\eta. \nonumber
\end{eqnarray}
We now apply integration by parts with respect to $\eta_1, \eta_2$, respectively, for $N =1,2, \dots, $   
\begin{eqnarray*}
\delta_{1,j} &\lesssim & 2^j
\sum_{\substack{ m_1, m_2 \neq 0, \\  |m| > 2^{\epsilon j} }} \Big |
\int_{\mathbb{R}^2} |m_1|^{-N} |m_2|^{-N} \frac{\partial^{2N}}{\partial \eta_1^{N}\partial \eta_2^{N}} \Big [\frac{1}{|\eta|}W(\eta ) \Big ] e^{2\pi i \zeta^T m} d\eta \Big | \\
& \leq  & C_{N} 2^j \sum_{|m|>2^{\epsilon j}} |m_1|^{-N}|m_2|^{-N} \\
& \leq  & C_{N} 2^j 2^{-(N-1)(\epsilon j-1)}, \quad \forall N \in \mathbb{N}.
\end{eqnarray*}
Here we note that the boundary terms vanish due to the compact support of $W$. 
This finishes the proof. 
\end{proof}
\begin{lemma} \label{3007-2}
We have $\delta_{2j} = \circ(2^{-Nj}), \quad j \rightarrow \infty.$ 
\end{lemma}
\begin{proof}
This lemma is a special case of  \cite[Proposition 5.2]{7}. It was used there universal shearlet systems with flexible scaling parameter $\alpha$ which coincide with our shearlets $\Psi$  for case $\alpha=1$.
\end{proof}

\begin{lemma} \label{3007-3}
For every $N \in \mathbb{N}$ the sequence $(\delta_{3,j})_{j \in \mathbb{N}}$ decays rapidly, i.e.,
\begin{equation}  
\delta_{3,j}:=\sum_{(m^\prime, n^\prime) \notin B(0,M_j) \times I_T^{\pm} } | \langle \mathcal{T}_{j}, g_{m^\prime, n^\prime} \rangle |= o(2^{-Nj}).
\end{equation}
\end{lemma}

\begin{proof}
The proof of this lemma follows an argument similar to that of \cite[Proposition 8.2]{14}. We also include the proof here since we used a modified model of texture. By definition, we have
\begin{eqnarray*}
   \delta_{3,j} = \sum_{(m^\prime, n^\prime) \notin B(0,M_j) \times I_T^{\pm} } |  \langle \mathcal{T}_{j}, g_{m^\prime, n^\prime} \rangle | 
\end{eqnarray*}
We now consider two cases.

Case 1: $n^\prime \notin I_T^\pm \cap \mathcal{A}_j$.
By Plancherel, we have
\begin{eqnarray*}
  |  \langle \mathcal{T}_{j}, g_{m^\prime, n^\prime} \rangle | 
 &=& |  \langle \hat{\mathcal{T}}_{j}, \hat{g}_{m^\prime, n^\prime} \rangle | \\
  &=& \sum_{ |n^\prime -n| >1} \int \Big | \sum_{n \in I_T } d_n W_j(\xi) \hat{g}(\xi - n) \hat{g}(\xi - n^\prime) e^{2 \pi i  \xi^\top \frac{m^\prime}{2}} \Big | \\
  & =& 0,
 \end{eqnarray*}
where the last equality is due to $\supp \hat{g} = [-1,1]^2,$ and $\supp W_j = \mathcal{A}_j.$

Case 2: $n^\prime \in I_T^\pm \cap \mathcal{A}_j$. 
By Lemma \ref{0907-2} and the boundedness of $(d_n)_{n \in \mathbb{Z}^2}$, we obtain 
\begin{eqnarray*}
   \delta_{3,j} &=& \sum_{\substack{m^\prime  \notin B(0,M_j) \\ n^\prime \in I_T^\pm \cap \mathcal{A}_j }} |  \langle \mathcal{T}_{j}, g_{m^\prime, n^\prime} \rangle | = \sum_{\substack{m^\prime  \notin B(0,M_j) \\ n^\prime \in I_T^\pm \cap \mathcal{A}_j }}  \Big | \int_{\mathbb{R}^2} \sum_{n \in I_T} d_n g_{0,n}(x)  g_{m^\prime, n^\prime}(x) dx \Big | \\
   &\leq &  \sum_{\substack{ |m^\prime| > 2^{\epsilon j/6} \\  n^\prime \in I_T^\pm \cap \mathcal{A}_j}}   \int_{\mathbb{R}^2}C_N  \langle | x_1|\rangle^{-N}\langle | x_2|\rangle^{-N} \langle | x_1 + \frac{m^\prime_1}{2}|\rangle^{-N} \langle | x_1 + \frac{m^\prime_2}{2}|\rangle^{-N} dx\\
   &\leq &  C_N  \cdot \sum_{\substack{ |m^\prime| > 2^{\epsilon j/6} \\ n^\prime \in I_T^\pm \cap \mathcal{A}_j }}      \langle | \frac{m^\prime_1}{2}|\rangle^{-N} \langle | \frac{m^\prime_2}{2}|\rangle^{-N},
\end{eqnarray*}
here the last equality is due to the fact that  there exists a constant $C_N>0$ satisfying
\begin{equation*}
    \int_{\mathbb{R}} \langle |t| \rangle^{-N} \langle | t + a| \rangle^{-N} dt
\leq  C_N \langle |a| \rangle^{-N}, \quad  \forall a \in \mathbb{R}. \label{100615}
\end{equation*}
Thus, by \eqref{1907} we have 
\begin{eqnarray*}
  \delta_{3,j} &\leq & C_N\cdot  2^{(1-\epsilon)j} \cdot  \sum_{|m^\prime| > 2^{\epsilon j/6} }\langle |\frac{m^\prime_1}{2}|\rangle^{-N}  \langle |\frac{m^\prime_2}{2}|\rangle^{-N}  \\
  &\leq & C_N \cdot 2^{(1-\epsilon)j} \cdot  \int_\mathbb{R} \int^{+\infty}_{\frac{1}{2} 2^{\epsilon j/6}} \langle |\frac{t_1}{2}|\rangle^{-N}  \langle | \frac{t_2}{2}|\rangle^{-N} dt_1 dt_2 \\
&  \leq & C_N^\prime \cdot 2^{(1-\epsilon)j} \cdot 2^{-(N-1)\epsilon j/6}.
\end{eqnarray*}
This concludes the claim since we can choose an  arbitrarily large $N \in \mathbb{N}.$  
\end{proof}

As a direct consequence of Lemmata \ref{3007-1}, \ref{3007-2}, \ref{3007-3}, we have the following proposition
\begin{proposition} \label{2107-20}
For any $N \in \mathbb{N}$, we have $\delta_{j}:= \delta_{1j}+ \delta_{2j}+ \delta_{3j} = \circ(2^{-Nj}).$
\end{proposition}
\subsection{Estimate of cluster coherence for separation task}
In this section, we give an estimate for the term $\mu_{c,N}^{sep}$.
The following lemma ensures the success of separating points from curves
\begin{proposition} \label{2107-21}
We have
\begin{enumerate} [label=(\roman*)]
    \item $\mu_c(\Lambda_{1,j}^{\pm}, \mathbf{\Psi}; \mathbf{\Phi}) \rightarrow 0, \; j \rightarrow \infty,$ \\
    \item $\mu_c(\Lambda_{2,j}^{\pm}, \mathbf{\Phi}; \mathbf{\Psi}) \rightarrow 0, \; j \rightarrow \infty. $
\end{enumerate}
\end{proposition}
\begin{proof}

(i)  By Lemma \ref{0907-2}, we have
\begin{eqnarray}
\mu_c(\Lambda_{1j}, \mathbf{\Psi}; \mathbf{\Phi}) &\leq & \# \{ (j,m) \in \Lambda_{1,j} \} \cdot  C \cdot 2^{-j/2} \nonumber \\
&\leq &C \cdot 2^{2\epsilon j} \cdot  2^{-j/2} \nonumber \\
&=& C \cdot 2^{-(1 -4 \epsilon)j/2} \xrightarrow[j \rightarrow +\infty ]{\textup{By } \eqref{0907-1}} 0. \nonumber
\end{eqnarray}

ii) Lemma \ref{0907-2} yields
\begin{eqnarray}
&& \mu_c(\Lambda_{2j}, \mathbf{\Phi}; \mathbf{\Psi}) \leq  C_N   2^{7j/2} \sum_{\substack{|l| \leq 1, |j^{\prime}-j| \leq 1 \\k \in \mathbb{Z}^2, |k_2 -  lk_1 | < 2^{\epsilon j^\prime} }} \int_{\mathbb{R}^2} \langle |2^{2j^{\prime}} x_1 + m^{\prime}_1|\rangle^{-N} \cdot \nonumber \\
&& \quad  \langle |2^{2j^{\prime}} x_2 + m^{\prime}_2|\rangle^{-N}  \langle |2^{ j^{\prime}}x_1 + k_1|\rangle^{-N} \langle |2^{2j^{\prime}}x_2+ l2^{ j^{\prime} } x_1+k_2|\rangle^{-N} dx. \nonumber
\end{eqnarray}
We consider only  $j^{\prime}=j$ since the cases are estimated similary up to a constant.

Using the change of variable $(y_1, y_2) = (2^{2j}x_1, 2^{2j}x_2)$ yields
\begin{eqnarray}
\mu_c(\Lambda_{2j}, \mathbf{\Phi}; \mathbf{\Psi}) &\leq & C_N   2^{-j/2} \sum_{\substack{l \in \{ -1, 0,1\}, k \in \mathbb{Z}^2 \\ |k_2 -  lk_1 | < 2^{\epsilon j^\prime} }}
\int_{\mathbb{R}^2} \langle |y_1 + m^{\prime}_1|\rangle^{-N}\langle |y_2 + m^{\prime}_2|\rangle^{-N}  \nonumber \\
&&  \qquad \cdot \langle |2^{-j}y_1 + k_1|\rangle^{-N} \langle |y_2+ l2^{ j } x_1+k_2|\rangle^{-N} dy. \label{466}
\end{eqnarray}
Furthermore, we have
\begin{equation}
    \sum_{ k \in \mathbb{Z}^2 }  \langle |2^{-j}y_1 + k_1|\rangle^{-N} \langle |y_2+ l2^{ j } x_1+k_2|\rangle^{-N}   \leq C^{\prime}. \label{467}
\end{equation}
Combining \eqref{466} with \eqref{467}, we obtain
\begin{eqnarray}
\mu_c(\Lambda_{2j}, \mathbf{\Phi}; \mathbf{\Psi}) &\leq & C^{\prime}_N \cdot   2^{-j/2}
\int_{\mathbb{R}^2} \langle |y_1 + m^{\prime}_1|\rangle^{-N}\langle |y_2 + m^{\prime}_2|\rangle^{-N} dy \nonumber \\
&\leq & C_N^{''} \cdot 2^{-j/2} \xrightarrow{j \rightarrow +\infty } 0.
\end{eqnarray}
We complete the proof.
\end{proof}

\begin{proposition} \label{2107-22}
We have 
\begin{enumerate}[label=(\roman*)]
    \item $\mu_c(\Lambda_{3,j},  \mathbf{G};\mathbf{\Psi}), \mu_c(\Lambda_{3,j},  \mathbf{G};\mathbf{\Phi}) \rightarrow 0, \quad j \rightarrow \infty.$ \\
    \item $\mu_c(\Lambda_{1,j}, \mathbf{\Phi}; \mathbf{G}), \mu_c(\Lambda_{2,j}, \mathbf{\Psi}; \mathbf{G}) \rightarrow 0, \quad j \rightarrow \infty.$
\end{enumerate}

\end{proposition}
\begin{proof}
i) We first consider the term $\mu_c(\Lambda_{3,j},  \mathbf{G};\mathbf{\Psi})$. For this, we assume  that the maximum is attained at $l^\prime \in \mathbb{Z}, k^\prime \in \mathbb{Z}^2$ and $\rm{\iota} \in \{ \rm{h}, \rm{v}  \}$. We now use \eqref{1907} and Lemma \ref{0907-2}, we have
\begin{eqnarray*}  
\mu_c(\Lambda_{3,j},  \mathbf{G}; \mathbf{\Psi}) & = & \sum_{m \in B(0, M_j), n \in I_T^{\pm}}  | \langle g_{m,n}, \psi_{j, l^\prime, k^\prime}^{ \iota} \rangle | \\
& \leq &  C_N \cdot 2^{- j}  \# \{ m \in B(0,M_j) \} \cdot \#  \{ n \in I_T^{\pm} \cap \mathcal{A}_j\} \\
& \leq & C_N \cdot  2^{-j}  \cdot M_j^2 \cdot 2^{(1-\epsilon)j}  \\
& \stackrel{\mathclap{\normalfont{M_j = 2^{\epsilon j/6} }}} =& \quad C_N \cdot 2^{-2\epsilon j/3}
\xrightarrow{j \rightarrow +\infty} 0. 
\end{eqnarray*}
Similarly, we have
\begin{eqnarray*}  
\mu_c(\Lambda_{3,j},  \mathbf{G}; \mathbf{\Phi}) & = & \sum_{m \in B(0, M_j), n \in I_T^{\pm}}  | \langle g_{m,n}, \phi_{j,  k^\prime} \rangle | \\
& \leq &  C_N \cdot 2^{- 2j}  \# \{ m \in B(0,M_j) \} \cdot \#  \{ n \in I_T^{\pm} \cap \mathcal{A}_j \} \\
& \leq & C_N \cdot  2^{-2j}  \cdot M_j^2 \cdot 2^{(1-\epsilon)j}  \\
& \stackrel{\mathclap{\normalfont{M_j = 2^{\epsilon j/6} }}} =& \quad C_N \cdot 2^{-j}\cdot 2^{-2\epsilon j/3}
\xrightarrow{j \rightarrow +\infty} 0. 
\end{eqnarray*}

ii)  By the definition
$$ \mu_c (\Lambda_{2,j},  \mathbf{\Psi}   ; \mathbf{G}) = \max_{(m,n)} \sum_{\eta \in \Lambda_{1,j}^\pm} |\langle  \psi_{j, l, k}^{ \rm{\iota}}, g_{m,n} \rangle|,$$ where $\eta=(j, l, k;
\rm{\iota}).$ 
We assume  that the maximum is achieved at  $m^{\prime}, n^{\prime} \in \mathbb{Z}^2.$  By Lemma \ref{0907-2} (i),(iii), and  the change of variables $(y_1, y_2) = ( x_1, 2^{2j}x_2+ l2^{ j } x_1 -lk_1)$, we have 
\begin{eqnarray*}
 \mu_c (\Lambda_{1,j},  \mathbf{\Psi}; \mathbf{G})  &=& \sum_{\eta \in \Lambda_{1,j}} |\langle  \psi_{j, l, k}^{ \rm{\iota}}, g_{m^{\prime},n^{\prime}} \rangle| \\
 & \leq &  C_N  2^{\frac{3}{2}j}  \sum_{\substack{|l| \leq 1, k \in \mathbb{Z}^2 \\ |k_2 - lk_1| \leq 2^{\epsilon j} }} \int\limits_{\mathbb{R}^2} \langle |2^{ j}x_1 - k_1|\rangle^{-N} \langle |2^{2j}x_2+ l2^{ j } x_1-k_2|\rangle^{-N} \\
&& \qquad 
 \cdot \langle | x_1 + \frac{m_1}{2}|\rangle^{-N}\langle | x_2 + \frac{m_2}{2}|\rangle^{-N}dx_1 dx_2. \\
 & =&   C_N  2^{-
 \frac{1}{2}j} \sum_{\substack{|l| \leq 1, k \in \mathbb{Z}^2 \\ |k_2 - lk_1| \leq 2^{\epsilon j} }} \int_{\mathbb{R}^2}\langle |2^{j}y_1 - k_1|\rangle^{-N} \langle |y_2+ lk_1-k_2|\rangle^{-N} \cdot \\
&&   \langle |y_1 + \frac{m_1}{2}|\rangle^{-N} \langle |2^{-2j} y_2 -l2^{-j} y_1+2^{-2j}lk_1 +\frac{m_2}{2}|\rangle^{-N}dy_1 dy_2 \\ 
& \leq &   \sum_{\substack{|l| \leq 1, k \in \mathbb{Z}^2\\ |k_2| \leq 2^{\epsilon j} }} \int_{\mathbb{R}^2}C_N \cdot 2^{-j/2}  \langle |2^{ j}y_1 - k_1|\rangle^{-N} \langle |y_2-k_2|\rangle^{-N} \\ 
&& \qquad \qquad \quad \cdot \langle |y_1 + \frac{m_1}{2}|\rangle^{-N} dy_1dy_2 \\
 & =&  \sum_{\substack{|l| \leq 1, k_2 \in \mathbb{Z}\\ |k_2| \leq 2^{\epsilon j} }} \int_{\mathbb{R}^2}C_N2^{-j/2}  \Big ( \sum_{k_1 \in \mathbb{Z}}\langle |2^{ j}y_1 - k_1|  \rangle^{-N} \Big ) \langle |y_2-k_2|\rangle^{-N} \\
 && \qquad \qquad 
 \cdot \langle |y_1 + \frac{m_1}{2}|\rangle^{-N}dy_1 dy_2 \\ 
 & \leq & C_N \cdot 2^{-j/2} \cdot 2^{\epsilon j} = C_N \cdot 2^{-(1-2\epsilon)j/2} \xrightarrow{j \rightarrow +\infty} 0. 
\end{eqnarray*}
Similarly, we assume that the maximum of the cluster coherence $\mu_c (\Lambda_{1,j},  \mathbf{\Phi}   ;  \mathbf{G})$ is attained for some $(m^{\prime \prime}, n^{\prime \prime}) \in \mathbb{Z}^2.$ By Lemma \ref{0907-2} (iv),  we have 
\begin{eqnarray*}
 \mu_c (\Lambda_{1,j},  \mathbf{\Phi}   ;  \mathbf{G})  &=&  \sum_{|p| \leq 2^{\epsilon j} } |\langle  \phi_{j,p}, g_{m^{\prime \prime},n^{\prime \prime}} \rangle| \\
 & \leq &  C_N \cdot  2^{-2j} \cdot 2^{2\epsilon j} = C_N \cdot 2^{-(2-\epsilon)j}  \xrightarrow{j \rightarrow +\infty} 0. 
\end{eqnarray*}
This concludes the claim. 
\end{proof}

The following lemma guarantees the  successful separation of curves and texture as well as points and texture.
\begin{proposition} \label{2107-23}
 We have
\begin{enumerate}[label=(\roman*)]
    \item  $\mu_c (\Lambda_{3,j}^{\pm},  P_{M,j} \mathbf{G}; \mathbf{\Psi})  \rightarrow 0, \quad j \rightarrow \infty.$
    \item $\mu_c (\Lambda_{2,j}^{\pm}, P_{M,j} \mathbf{\Psi};  \mathbf{G})  \rightarrow 0, \quad j \rightarrow \infty.$
\end{enumerate}
\end{proposition}
\begin{proof}
i) By definition, we have
 \begin{equation*}
     \mu_c (\Lambda_{3,j}, P_{M,j}   \mathbf{G}; \mathbf{\Psi}) = \sup_{(j,l,k)} \sum_{(m,n) \in\Lambda_{3,j}^\pm}| \langle P_{M,j} g_{m,n},  \psi_{j, l, k}^{ \rm{\iota}}\rangle |.
 \end{equation*}
Without loss of generality, we may assume that the maximum is achieved at $(j,l^\prime,k^\prime; \rm{\iota})\in \Delta_j$. Thus, we obtain
\begin{eqnarray*} \mu_c (\Lambda_{3,j}, P_{M,j}   \mathbf{G}; \mathbf{\Psi}) &=& \sum_{|m| \leq 2^{\epsilon/6}, n \in I_T^{\pm} \cap \mathcal{A}_{j}}| \langle P_{M,j} g_{m,n},  \psi_{j, l^\prime, k^\prime}^{ \rm{\iota}}\rangle |\\
& =& \sum_{|m| \leq 2^{\epsilon/6}, n \in I_T^{\pm} \cap \mathcal{A}_{j}} \Big | \int_{-h_j}^{h_j} \int_\mathbb{R} g_{m,n}(x) \overline{\psi_{j, l^\prime, k^\prime}^{ \rm{\iota}}(x)} dx_1dx_2 \Big | \\
& \leq &  \sum_{|m| \leq 2^{\epsilon/6}, n \in I_T^{\pm} \cap \mathcal{A}_{j}} \int_{-h_j}^{h_j} \int_\mathbb{R} |g_{m,n}(x)| |\psi_{j, l^\prime, k^\prime}^{ \rm{\iota}}(x)| dx_1dx_2. 
\end{eqnarray*}
We consider only the case $\rm{\iota} = \rm{v},$ the other case can be done analogously. Indeed, we recall Lemma \ref{0907-2} (i),(iii),
$$ |g_{m,n}(x)| \leq  C_N  \cdot  \langle |x_1 + \frac{m_1}{2}|\rangle^{-N}\langle | x_2 + \frac{m_2}{2}|\rangle^{-N} , \forall m, n \in \mathbb{Z}^2,$$
$$ | \psi_{j, l^\prime, k^\prime}^{ \rm{v}}(x)| \leq C_N \cdot 2^{3j/2} \cdot \langle |2^{ j}x_1 - k^\prime_1|\rangle^{-N} \langle |2^{2j}x_2+ l^\prime 2^{ j } x_1-k^\prime_2|\rangle^{-N}.$$
We now apply the  change of variable $y =S^{l^\prime}_{\rm{v}}A_{ \rm{v}}^{j} x= (2^{ j}x_1, 2^{2j}x_2 + l^\prime2^{ j}x_1)$. This leads to
\begin{eqnarray}
 \int_{-h_j}^{h_j} \int_\mathbb{R}  |g_{m,n}(x)| | \psi_{j, l^\prime, k^\prime}^{ \rm{v}}(x)|   dx 
 &\leq &  C_N 2^{-\frac{3}{2}j}    \int_{\mathbb{R}^2}     \langle| 2^{- j}y_1 + \frac{m_1}{2} |\rangle^{-N} \langle |y_1-k^{\prime}_1|\rangle^{-N}   \nonumber \\
 && \cdot \langle | 2^{-2j}(y_2-l^\prime y_1) + \frac{m_2}{2}|\rangle^{-N}  \langle |y_2- k^\prime_2|\rangle^{-N} dy \nonumber \\
  &\leq &  C_N2^{-\frac{3}{2}j} \int_{\mathbb{R}^2}     \langle| 2^{- j}y_1 + \frac{m_1}{2} |\rangle^{-N} \langle |y_1-k^\prime_1|\rangle^{-N}  \nonumber \\
  &&  \cdot \langle | 2^{-2j}(y_2-l^\prime y_1) + \frac{m_2}{2}|\rangle^{-N}   \langle |y_2- k^\prime_2|\rangle^{-N} dy. \nonumber \\ \label{2107-1}
 \end{eqnarray}
In addition, we have 
\begin{equation} \label{2107-2}
     \sum_{(m_1, m_2) \in \mathbb{Z}^2} \langle | 2^{- j}y_1 + \frac{m_1}{2} |\rangle^{-N} \langle | 2^{-2j}(y_2-l^\prime y_1)  + \frac{m_2}{2} |\rangle^{-N} \leq C_N^\prime 
\end{equation}
 and
\begin{equation} \label{2107-3}
    \int_{\mathbb{R}^2} \langle |y_1 - k^\prime_1 |\rangle^{-N} \langle |y_2 - k^\prime_2 |\rangle^{-N}  dy_1 dy_2 \leq C_N^{\prime \prime}.
\end{equation}
Combining (\ref{2107-1}), (\ref{2107-2}), (\ref{2107-3})  and (\ref{1907}), we finally obtain
\begin{eqnarray*}
 \mu_c (\Lambda_{3,j}, P_{M,j}   \mathbf{G}; \mathbf{\Psi})  &\leq & \sum_{n \in I_T^{\pm} \cap \mathcal{A}_{j}} C_N \cdot  2^{-3j/2} \cdot   \\
& \leq & C_N \cdot  2^{-3j/2} \cdot 2^{j}  =   C_N \cdot  2^{-j/2} \xrightarrow  {j \rightarrow +\infty} 0. 
\end{eqnarray*}  

ii) For the other term $\mu_c (\Lambda_{2,j}^\pm, P_{M,j} \mathbf{\Psi};\mathbf{G}),$  we assume that the maximum in the definition of the cluster is attained for some $(m^\prime, n^\prime) \in \mathbb{Z}^2 \times \mathbb{Z}^2  $. Thus, we have
\begin{eqnarray*} \mu_c (\Lambda_{1,j}, P_{M,j} \mathbf{\Psi}; \mathbf{G} ) &=&  \sum_{(j, l, k; 
\iota) \in \Lambda_{1,j}} | \langle P_{M,j} \psi_{j, l, k}^{
\rm{
\iota
}} , g_{m^\prime,n^\prime}\rangle |\\
& =& \sum_{\substack{|l| \leq 1, k \in \mathbb{Z}^2 \\ |k_2 - lk_1| \leq 2^{\epsilon j} }}  \Big | \int_{-h_j}^{h_j} \int_\mathbb{R} \psi_{j, l, k}^{
\rm{v}}(x) \overline{g_{m^\prime,n^\prime}(x)}  dx \Big | \\
& \leq &  \sum_{\substack{|l| \leq 1, k \in \mathbb{Z}^2 \\ |k_2 - lk_1| \leq 2^{\epsilon j} }}   \int_{-h_j}^{h_j} \int_\mathbb{R} |\psi_{j, l, k}^{\rm{v}}(x)| |g_{m^\prime,n^\prime}(x)|  dx|. 
\end{eqnarray*}
 By Lemma \ref{0907-2}
 and the change of variable $y =S^{l}_{\rm{v}}A_{ \rm{v}}^{j} x= (2^{ j}x_1, 2^{2j}x_2 + l2^{ j}x_1)$, we obtain
\begin{eqnarray}
  \int\limits_{\text{-}h_j}^{h_j} \int\limits_\mathbb{R}  | \psi_{j, l, k}^{ \rm{v}}(x)|| g_{m^\prime,n^\prime}(x)|   dx  &\leq &  C_N  2^{\frac{\text{-}3j}{2}}   \int\limits_{\text{-}2^jh_j}^{2^jh_j} \int\limits_{\mathbb{R}}   \langle |y_1-k_1|\rangle^{-N} \langle| 2^{-
 j}y_1 + \frac{m^\prime_1}{2} |\rangle^{- N}    \nonumber \\
 &&  \langle |y_2- k_2|\rangle^{-N} 
 \langle | 2^{-2j}(y_2-ly_1) + \frac{m^\prime_2}{2}|\rangle^{-N}  dy. \nonumber \\
 \label{2107-11}
 \end{eqnarray}
Furthermore, we have
\begin{equation} \label{2107-12}
    \langle | 2^{-
    j}y_1 + \frac{m^\prime_1}{2} |\rangle^{-N} \langle | 2^{-2j}(y_2-ly_1)  + \frac{m^\prime_2}{2} |\rangle^{-N} \leq 1
\end{equation} 
 and
 \begin{eqnarray}
&& \int_\mathbb{R} \Big ( \sum_{\substack{|l| \leq 1, k \in \mathbb{Z}^2 \\ |k_2 - lk_1| \leq 2^{\epsilon j} }}\langle |y_1-k_1|\rangle^{-N}  \langle |y_2- k_2|\rangle^{-N} \Big ) dy_2  \nonumber \\
& = &  \sum_{\substack{|l| \leq 1, k_1, k_3 \in \mathbb{Z} \\ |k_3| \leq 2^{\epsilon j} }}\langle |y_1-k_1|\rangle^{-N} \cdot \Big ( \int_\mathbb{R} \langle |y_2- k_3-lk_1|\rangle^{-N} dy_2 \Big )\nonumber \\
 &\leq & C_N \cdot 2^{\epsilon j} \cdot \int_{\mathbb{R}} \langle |t-y_1|\rangle^{-N}  dt \leq C_N^\prime 2^{\epsilon j}. \label{2107-13}
 \end{eqnarray}
We now combine (\ref{2107-11}), (\ref{2107-12}) and  (\ref{2107-13}) to obtain
\begin{eqnarray*}
 \mu_c (\Lambda_{2,j}, P_j \mathbf{\Psi}; \mathbf{G} ) &\leq& C_N^{\prime \prime} \cdot 2^{-3j/2} \cdot 2^{ j} h_j \cdot 2^{\epsilon j} \\
& \leq& \;\; C_N^{\prime \prime} \cdot 2^{-(1-2\epsilon)j/2} \cdot 2^{ j} h_j \xrightarrow{j \rightarrow +\infty} 0 .
\end{eqnarray*}
\end{proof} 

Next is our final estimate.
\subsection{Estimate of cluster coherence for inpainting task}

The previous subsection ensures the success of separating three components. Now it remains to provide estimates for inpainting curves as well as texture. Indeed, by our assumption that the pointlike part has no missing content so we need only to inpaint curvilinear and texture structures. For this, we would like to recall that the inpainting task is encoded by the term  $\mu_{c,N}^{inp}:= \max\limits_{1 \leq m \leq N} \inf\limits_{\Phi_m^d \in \mathbb{D}_{\Phi_m}}  \mu_c(\Lambda_m, P_M \Phi_m; \Phi_m^d )$. The following lemma is what we need.

\begin{proposition} \label{2107-24}
 We have
\begin{enumerate}[label=(\roman*)]
    \item $\mu_c (\Lambda_{2,j}^{\pm}, P_{M,j} \mathbf{\Psi} ; \mathbf{\Psi}  )  \rightarrow 0, \; j \rightarrow \infty.$
    \item $\mu_{c}(\Lambda_{3,j}, P_{M,j}\mathbf{G}; \mathbf{G}) \rightarrow 0, \; j \rightarrow \infty.$
\end{enumerate}
\end{proposition} 
\begin{proof}
Here we would like to remark again that universal shearlets  used in \cite{14} coincide with our shearlet system $\Phi$ in case $\alpha=1.$ Now it remains to estimate the other term $\mu_{c}(\Lambda_{3,j}, P_{M,j}\mathbf{G}; \mathbf{G})$.

Without loss of generality, we assume that the maximum is achieved at $(m^\prime, n^\prime) \in \Lambda_{3, j}$. Thus, we obtain
\begin{eqnarray}  \mu_c (\Lambda_{3,j}, P_{M,j}   \mathbf{G};   \mathbf{G} ) &=& \sum_{(m,n) \in M_j \times (I_T^{\pm} \cap \mathcal{A}_{j})}| \langle P_{M,j} g_{m,n}, g_{m^\prime,n^\prime}\rangle | \nonumber \\
& =& \sum_{(m,n) \in M_j \times (I_T^{\pm} \cap \mathcal{A}_{j})} \Big | \int_{-h_j}^{h_j} \int_\mathbb{R} g_{m,n}(x) \overline{g_{m^\prime,n^\prime}(x)} dx \Big | \nonumber \\
& \leq &  \sum_{(m,n) \in M_j \times (I_T^{\pm} \cap \mathcal{A}_{j})} \int_{-h_j}^{h_j} \int_\mathbb{R} |g_{m,n}(x)| |g_{m^\prime,n^\prime}(x)| dx. \qquad \label{1907-A1}
\end{eqnarray}
Next, we recall Lemma \ref{0907-2} (i),
$$ |g_{m,n}(x)| \leq  C_N  \cdot  \langle | x_1 + \frac{m_1}{2}|\rangle^{-N}\langle | x_2 + \frac{m_2}{2}|\rangle^{-N}, \forall m,n \in \mathbb{Z}^2.$$
This leads to
\begin{eqnarray}
  \int_{-h_j}^{h_j} \int_\mathbb{R}  |g_{m,n}(x)|  |g_{m^\prime,n^\prime}(x)| dx 
 &\leq&  \int_{-h_j}^{h_j} \int_\mathbb{R}  C_N ^2  \langle| x_1 + \frac{m_1}{2} |\rangle^{-N} \langle | x_2 + \frac{m_2}{2} |\rangle^{-N}  \cdot
\nonumber \\
&& \langle| x_1 + \frac{m_1^\prime}{2} |\rangle^{-N}  \cdot \langle | x_2 + \frac{m_2^\prime}{2} |\rangle^{-N} dx  \nonumber \\
 & = & C_N^2\int_{-h_j}^{h_j} \int_\mathbb{R}   \langle |y_1 + \frac{m_1}{2} |\rangle^{-N} \langle | y_2 + \frac{m_2}{2} |\rangle^{-N} \cdot \nonumber \\
 &&  \langle |y_1 + \frac{m_1^\prime}{2} |\rangle^{-N} \langle |y_2 + \frac{m_2^\prime}{2} |\rangle^{-N} dy. 
 \label{1907-A2}
 \end{eqnarray}
In addition, we have 
\begin{equation}
    \sum_m \langle |y_1 + \frac{m_1}{2} |\rangle^{-N} \langle | y_2 + \frac{m_2}{2} |\rangle^{-N} \leq C_N^\prime, \; \int_\mathbb{R} \langle |y_2 + \frac{m_2^\prime}{2} |\rangle^{-N} dy_2 \leq C_N^{\prime \prime}. \label{1907-A3}
\end{equation} 
 Thus, by (\ref{1907}) and (\ref{1907-A1}) combined with (\ref{1907-A2}) and (\ref{1907-A3}) , we obtain
 \begin{eqnarray*}
  \mu_c (\Lambda_{2,j}, P_j \mathbf{G}; \mathbf{G} ) &\leq & \sum_{n \in I_T^{\pm} \cap \mathcal{A}_{j}} C_N^{\prime \prime \prime} \cdot h_j   \\
  & \leq & C_N^{\prime \prime \prime} \cdot 2^{ j }h_j\xrightarrow  {j \rightarrow +\infty} 0.
 \end{eqnarray*}
 This concludes the claim.
\end{proof} 

Let us go back to the proof of our main theorems.
\subsection{Proof of Theorem \ref{3007-10} and Theorem \ref{3007-11}}
\begin{proof}
Applying Theorem \ref{2006-3} with respect to Remark \ref{3007-21} (5), we obtain
$$
\| \mathcal{P}_j^* - \mathcal{P}_j  \|_2+ \| \mathcal{C}_j^* - w\mathcal{L}_j  \|_2+ \| \mathcal{T}_j^* - \mathcal{T}_{j}  \|_2 \leq \frac{\delta_{1j}+ \delta_{2j}+ \delta
_{3j}}{1-2\mu_{c,3}}, $$
where $\mu_{c, 3}$ is bounded by \eqref{1308-1}. Roughly speaking, we just need to inpaint  two components of curvilinear singularities and texture and simultaneously separate three geometric components including point singularities.
By Proposition \ref{2107-20}, we have $\delta=\delta_{1,j}+ \delta_{2,j}+ \delta_{3,j} = \circ(2^{-Nj}), \; \forall N \in \mathbb{N}.$
In addition, Propositions \ref{2107-21}, \ref{2107-22}, \ref{2107-23}  yield $\mu_{c,3} \rightarrow 0, \; j \rightarrow \infty.$
This concludes the claim.
\end{proof}

The proof of Theorem \ref{3007-11} follows the same lines as  Theorem \ref{3007-10}. Here we use Theorem \ref{2006-4} instead of Theorem \ref{2006-3}. Now we need to verify \eqref{2006-5}, i.e., 
\begin{equation}
    \sum_{(j,p) \in \Lambda_{1,j}} | \langle \phi_{j,p}, z \rangle | +  \sum_{(j,l,k) \in \Lambda_{2,j}} | \langle \psi_{j,l,k}, z \rangle | + \sum_{(j,p) \in \Lambda_{3,j}} | \langle g_{m,n}, z \rangle | \leq 2^{2j} \| z \|_1, \label{1308-2}
\end{equation}
for all $z \in L^2(\mathbb{R}^2).$
Indeed, by Lemma \ref{0907-2} we have
\begin{eqnarray}
  \sum_{(j,p) \in \Lambda_{1,j}} | \langle \phi_{j,p}, z \rangle | & \leq &\sum_{\substack{p \in \mathbb{Z}^2, \\ |p| \leq 2^{\epsilon j} }}  C_N 2^{2j} \int\limits_{\mathbb{R}^2}    \langle |2^{2j} x_1 + p_1|\rangle^{-N}\langle |2^{2j} x_2 + p_2|\rangle^{-N} |z(x)| dx \nonumber \\
  & \leq & C_N^\prime 2^{2j} \|z \|_1, \label{1308-3}
\end{eqnarray}
where the last inequality follows from
\begin{equation*}
    \sum_{\substack{p \in \mathbb{Z}^2, \\ |p| \leq 2^{\epsilon j} }}  \langle |2^{2j} x_1 + p_1|\rangle^{-N}\langle |2^{2j} x_2 + p_2|\rangle^{-N}  \leq C_N^{\prime \prime }, \; \forall x=(x_1, x_2) \in \mathbb{R}^2.
\end{equation*}
Similarly, we obtain
\begin{eqnarray}
   \sum_{(j,l,k) \in \Lambda_{2,j}} | \langle \psi_{j,l,k}^{\rm{v}}, z \rangle | & \leq &  \sum_{\substack{l \in \{  -1, 0,1 \}  \\   |k_2-lk_1| \leq 2^{\epsilon j} }} C_N  2^{3j/2} \int_{\mathbb{R}^2}   \langle |2^{j}x_1 - k_1|\rangle^{-N} \cdot \nonumber \\ 
   &&\langle|2^{2j}x_2+ l2^{ j } x_1-k_2|\rangle^{-N} |z(x)| dx\nonumber \\ 
   &\leq& C_N^\prime 2^{3j/2} \|z \|_1, \label{1308-4}
\end{eqnarray}
and 
\begin{eqnarray}
 \sum_{(m,n) \in \Lambda_{3,j}} | \langle g_{m,n}, z \rangle | & \leq & \sum_{\substack{|m| \leq 2^{\epsilon j/6} \\  n \in I_T^{\pm} \cap \mathcal{A}_j} }   C_N \int\limits_{\mathbb{R}^2}  \langle | x_1 + \frac{m_1}{2}|\rangle^{-N}\langle |x_2 + \frac{m_2}{2}|\rangle^{-N} |z(x)|dx \nonumber \\
 &\stackrel{\mathclap{\normalfont{\textup{By} \;\eqref{1907} }}} \leq \quad & C_N^\prime 2^{(1-\epsilon)j} \|z\|_1. \label{1308-5}
\end{eqnarray}
We conclude the claim \eqref{1308-2} by a combination of \eqref{1308-3},   \eqref{1308-4}, and  \eqref{1308-5}.

\section{Summary and conclusion} \label{0208-1}

In this paper, we provide a  theoretical study of the simultaneous separation, inpainting, and denoising 
for multi-component signals. To achieve this, we first modify the notions of joint concentration and cluster coherence to encode the separating and inpainting tasks separately. One guarantees the success of the geometric separation  and one for inpainting missing content. This is possible since the missing trip is assumed to be small which does not affect the geometry of the whole signal.
Theoretical results are shown in  Theorems \ref{2006-3} and \ref{2006-4} guarantee the success of our two proposed algorithms, one for constrained $l_1$ optimization and the other for unconstrained $l_1$ optimization.

Motivated by the theoretical guarantees,  we aim to decompose the corrupted image into pointlike, curvelike, and texture parts as accurately as possible. In our analysis, we use a multi-scale approach that the task of simultaneous separation and inpainting is done separately at each scale $j$ by Algorithm \ref{AL1-3007} and \ref{AL2-3007} using  wavelets, shearlets, and Gabor frames. The whole image is then reconstructed from its sub-images by the reconstruction  formula \eqref{0408-1}.  The main idea for the success of the proposed algorithms is  that the significant coefficients by wavelets, shearlets, and Gabor are clustered geometrically in phase space by microlocal analysis. The powerful tool  based on $l_1$ minimization then enables us to explore the differences among these underlying components.  For our main results, Theorem \ref{3007-10} and Theorem \ref{3007-11} show that at sufficiently fine scales the true components of the corrupted image are almost captured by proposed algorithms.
In applications, our theory can be adapted with the change of components and sparse representation systems as long as they provide negligible cluster sparsity and cluster coherence.
By choosing suitable sparsifying systems our methodology can handle the following tasks for multi-component images: 
\begin{enumerate}
    \item Simultaneous separation, inpainting, and denoising: In this case, Algorithm \ref{AL2-3007} is taken into account.
    \item Simultaneous separation and inpainting: For this purpose, we can apply Algorithm \ref{AL1-3007}.
    \item Inpainting:  Once the inpainted components are extracted, one can aggregate them again for inpainting purposes only.
    \item Image separation: In this situation, we restrict our results to the case of no missing parts in the original image.
\end{enumerate}


\end{document}